\documentclass{amsart}
\usepackage[foot]{amsaddr}
\usepackage{enumerate}
\usepackage{caption,subcaption}
\usepackage{ifthen}

\usepackage{amsmath,amsthm,amssymb,mathtools,amsfonts}
\usepackage{pgf}
\usepackage{graphicx}
\usepackage{tikz}
\usetikzlibrary{backgrounds, patterns, shapes, calc, decorations.pathreplacing, decorations.pathmorphing}
\tikzset{>=latex} 
\usepackage{pgfplots} 

\tikzset{%
    vtx/.style={draw,circle,very thin,white,fill=black,inner sep=0pt,text width=8pt},
    smvtx/.style={draw,circle,ultra thin,white,fill=black,inner sep=0pt,text width=5pt},
    edg/.style={thick},
    plt/.style={thick},
}
\tikzset{/pgf/foreach/parse=true}
\pgfdeclarelayer{backback}
\pgfdeclarelayer{back}
\pgfdeclarelayer{fore}
\pgfdeclarelayer{forefore}
\pgfsetlayers{backback,back,main,fore,forefore}

\newcommand{\tnct}{{\textstyle\binom{n}{2}}}
\newcommand{\defeq}{\stackrel{\rm def}{=}}
\newcommand{\ds}{\displaystyle}
\newcommand{\ts}{\textstyle}
\newtheorem{theorem}{Theorem}
\newtheorem{lemma}[theorem]{Lemma}
\newtheorem{corollary}[theorem]{Corollary}
\newtheorem{proposition}[theorem]{Proposition}
\newtheorem{definition}[theorem]{Definition}

\newtheorem{conjecture}[theorem]{Conjecture}
\newtheorem{question}[theorem]{Question}


\title{The edit distance of word-representable and comparability graphs}

\author{Sergey Kitaev}
\address{Department of Mathematics and Statistics,
University of Strathclyde, 26 Richmond Street,
Glasgow G1 1XH, UK}
\email{sergey.kitaev@strath.ac.uk}

\author{Ryan R. Martin}
\address{Department of Mathematics,
Iowa State University,
Ames, IA 50011, USA}
\email{rymartin@iastate.edu}

\newcommand{\hword}{{\mathcal H}_{\rm word}}
\newcommand{\hkword}[1]{{\mathcal H}_{#1{\rm -word}}}
\newcommand{\hcomp}{{\mathcal H}_{\rm comp}}
\newcommand{\ed}{{\rm ed}}
\newcommand{\forb}{{\rm Forb}}

\newcommand{\FF}{{\mathcal F}}
\newcommand{\HH}{{\mathcal H}}
\newcommand{\KK}{{\mathcal K}}

\newcommand{\RR}{{\mathcal R}}
\newcommand{\dist}{{\rm dist}}
\newcommand{\x}{{\bf x}}
\newcommand{\y}{{\bf y}}

\newcommand{\VW}{{\rm VW}}
\newcommand{\VB}{{\rm VB}}
\newcommand{\EW}{{\rm EW}}
\newcommand{\EB}{{\rm EB}}
\newcommand{\EG}{{\rm EG}}
\newcommand{\dg}{{\rm d}_{\rm G}}

\begin{document}

\maketitle

\begin{abstract}
    In this paper, we establish that the maximum edit distance of an $n$-vertex graph from the hereditary property of word-representable graphs is $n^2/8-o(n^2)$. In addition, we establish that the maximum edit distance of an $n$-vertex graph from the hereditary property of poset comparability graphs is $5n^2/32-o(n^2)$. 
    
    In fact, we determine the edit distance function over all edge densities $p\in [0,1]$ for the property of word-representable graphs, for the property of $k$-word-representable graphs for each $k\geq 2$, and for the property comparability graphs. The latter has a peculiar structure that requires an infinite sequence of colored regularity graphs. 
\end{abstract}

\section{Introduction}

\subsection{Word representation of graphs}
Word representations of graphs enable one to encode a graph as a string and were introduced by Pyatkin and the first author~\cite{KP}. A \emph{word-representation} of a graph $G=(V,E)$ is a word over the alphabet $V$ such that vertices $v_1$ and $v_2$ are adjacent in $G$ if and only if they alternate in the word. A graph $G$ is said to be \emph{$k$-word-representable} if there exists a word representing $G$ that contains exactly $k$ copies of each letter, that is, a $k$-uniform word. It was shown in~\cite{KP} that a graph is word-representable if and only if it is $k$-word-representable for some positive integer $k$. This allows one to define the \emph{representation number} of a word-representable graph $G$, denoted by $\RR(G)$, as the smallest positive integer $k$ such that $G$ is $k$-word-representable. By definition, $\RR(G)=\infty$ for any non-word-representable graph.

\subsection{Comparability graphs}
The \emph{comparability digraph} of a finite poset $P$, denoted $D_P$ is a directed graph where the vertex set is the elements of $P$ and $x\rightarrow_{D_P}y$ is a directed edge if and only if $x<_P y$.
The \emph{comparability graph} of a finite poset $P$ is simply the underlying graph of its comparability digraph.

There is a strong connection between comparability graphs and word-representable graphs, as indicated in~\cite{KP}. Namely, if a graph is word-representable, then the neighbourhood of any vertex in it must be a comparability graph. The converse of this statement is not true \cite{kitaev2015}. It is known~\cite{KS08} that every comparability graph can be represented by a concatenation of permutations.

\subsection{Hereditary properties of graphs}

A property of graphs is \emph{hereditary} if it is closed under isomorphism and deleting vertices. 
That is, if $H$ is in a hereditary property $\HH$, then every induced subgraph of $H$ is also in $\HH$. 

For a graph $F$, let $\forb(F)$ denote the hereditary property of graphs with no induced copy of $F$. 
For a set of graphs $\FF$, let $\forb(\FF)$ denote the hereditary property of graphs with no induced copy of $F$ for any $F\in\FF$. 
That is, $\forb(\FF)=\bigcap_{F\in\FF}\forb(F)$.
In fact, for every hereditary property $\HH$, there exists a (possibly infinite) set $\FF(\HH)$ such that $\HH=\bigcap_{F\in\FF(\HH)}\forb(F)$.

A hereditary property is \emph{trivial} if it is finite, otherwise it is \emph{nontrivial}. 
By Ramsey theory, a hereditary property $\HH$ is trivial if and only if there are positive integers $a$ and $b$ such that $\HH\subseteq\forb\bigl(K_a\bigr)\cap\forb\bigl(\overline{K_b}\bigr)$.

\subsection{The edit distance in graphs}
The \emph{edit distance between graphs $G$ and $H$} on the same labeled vertex set of size $n$ is 
$$\dist(G,H)\defeq\bigl|E(G)\triangle E(H)\bigr|/\tnct ,$$
where, in this case the $\triangle$ notation denotes the symmetric difference.
The \emph{edit distance between a graph $G$ and hereditary property $\HH$} is 
$$\dist(G,\HH)\defeq\min\bigl\{\dist(G,H) : H\in\HH, V(H)=V(G)\bigr\}.$$
The \emph{edit distance of a hereditary property $\HH$} is 
$$\dist(n,\HH)\defeq\max\bigl\{\dist(G,\HH) : |V(G)|=n\bigr\}.$$ 

The study of the edit distance in graphs was developed simultaneously by Axenovich, K\'ezdy and the second author~\cite{AKM}, by Alon and Stav~\cite{AS}, and (in a different context) by Marchant and Thomason~\cite{MT}.

The edit distance function was introduced by Balogh and the second author~\cite{BM}.
\begin{definition}[Balogh-M.~\cite{BM}]
    The \emph{edit distance function} of a hereditary property $\HH$ is a function of $p\in [0,1]$, is denoted $\ed_{\HH}(p)$, and is defined to be
    $$
        \ed_{\HH}(p) = \lim_{n\rightarrow\infty} \max\bigl\{\dist\bigl(G,\HH\bigr) : \bigl|V(G)]\bigr|=n, \bigl|E(G)\bigr|=\bigl\lfloor p\tnct\bigr\rfloor\bigr\} .
    $$
\end{definition}
The limit was proven to exist in~\cite{BM}.

\begin{theorem}[Balogh--M.~\cite{BM}]
    Let $\HH$ be a nontrivial hereditary property.
    If $G(n,p)$ denotes the binomial Erd\H{o}s--R\'enyi random graph~\cite{ER}\footnote{In fact, this binomial model of the random graph is due to Gilbert~\cite{Gil}.}, then
    $$
        \ed_{\HH}(p) = \lim_{n\rightarrow\infty} \mathbb{E}\bigl[\dist\bigl(G(n,p),\HH\bigr)\bigr].
    $$
    Furthermore, for every $p\in [0,1]$, the quantity $\dist\bigl(G(n,p),\HH\bigr)$ is concentrated around its mean. 
    That is, for every $\epsilon>0$,
    $$
    \mathbb{P}\Bigl(\bigl|\dist\bigl(G(n,p),\HH\bigr)
    -\mathbb{E}\bigl[\dist\bigl(G(n,p),\HH\bigr)\bigr]\bigr|
    >\epsilon\Bigr)\rightarrow 0
    $$
    as $n\rightarrow\infty$.
\end{theorem}

Let $\hword$ denote the hereditary property of word-representable graphs.
For a positive integer $k$, let $\hkword{k}$ denote the hereditary property of $k$-word-representable graphs.
Note that, according to the way the term is defined, an induced subgraph of a $k$-word-representable graph also has a $k$-representation, thus $\hkword{k}$ is, indeed, a hereditary property.
Let $\hcomp$ denote the hereditary property of comparability graphs.

Hefty, Horn, Muir, and Owens~\cite{HHMO} established that $\dist\bigl(n,\hword\bigr)=\Omega\bigl(n^2\bigr)$. 
In this paper, we asymptotically resolve the value of $\dist\bigl(n,\hword\bigr)$ by proving the following:
\begin{theorem}
    For all $p\in [0,1]$, $\ed_{\hword}(p)=\min\bigl\{p/3,1-p\bigr\}$. 
    Consequently, $\dist\bigl(n,\hword\bigr)=\tfrac{1}{4}\tnct-o\bigl(n^2\bigr)$. 
    Moreover, if $\dist\bigl(G,\hword)\geq\bigl(\tfrac{1}{4}-\epsilon\bigr)\tnct$, then $\bigl|E(G)\bigr|/\tnct\in \bigl[\tfrac{3}{4}-3\epsilon,\tfrac{3}{4}+\epsilon\bigr]$.
    \label{thm:word}
\end{theorem}

We also know the edit distance function for $\hkword{k}$, the property of being $k$-word-representable.
\begin{theorem}
    Let $k\geq 2$. 
    For all $p\in [0,1]$, $\ed_{\hkword{k}}(p)=\min\bigl\{p,1-p\bigr\}$.
    Consequently, $\ds \dist\bigl(n,\hkword{k}\bigr)=\tfrac{1}{2}\tnct-o\bigl(n^2\bigr)$. 
    Moreover, if $\dist\bigl(G,\hkword{k})\geq\bigl(\tfrac{1}{2}-\epsilon\bigr)\tnct$, then $\bigl|E(G)\bigr|/\tnct\in \bigl[\tfrac{1}{2}-\epsilon,\tfrac{1}{2}+\epsilon\bigr]$.
    \label{thm:kword}
\end{theorem}

We note that the result is the same for any $k\geq 2$.
As far as $\hkword{1}$, the only 1-word-representable graphs are complete graphs and so $\dist\bigl(G,\hkword{1}\bigr)=\binom{n}{2}-|E(G)|$ for any graph $G$ on $n$ vertices.

Finally, we address the hereditary property of comparability graphs, denoted $\hcomp$. 

\begin{definition}
    For every integer $k\geq 2$, let $p_k$ satisfy $\frac{2p_k-1}{1-p_k}=2\cos\bigl(\frac{2\pi}{k+1}\bigr)$. 
    \label{defn:pk}
\end{definition}
Equivalently, $p_k=1-1/\bigl(2\cos\bigl(\frac{2\pi}{k+1}\bigr)+2\bigr)$.
Observe that $p_2=0$, $p_3=1/2=0.5$, $p_4=\bigl(\sqrt{5}-1\bigr)/2\approx 0.618$, $p_5=2/3\approx 0.667$, $p_6\approx 0.692$, $p_7=\sqrt{2}/2\approx 0.707$ and so on. 
Furthermore the sequence $\{p_k\}_{k\geq 2}$ increases to a limit of $3/4$. 
\begin{theorem}
    There is a sequence of functions $g_2(p)=p/2, g_3(p)=\frac{p}{1+2p}, g_4(p)=\frac{-1+3p-p^2}{6p-2}, g_5(p)=\frac{-2+4p-p^3}{-4+6p+3p^2}, g_6(p), g_7(p), \ldots$ such that for all $p\in [0,1]$, 
    $$\ts \ed_{\hcomp}(p)=\min\bigl\{\frac{p}{2}, \frac{p}{1+2p}, \frac{-1+3p-p^2}{6p-2}, \frac{-2+4p-p^3}{-4+6p+3p^2}, g_6(p), g_7(p), \ldots, 1-p\bigr\}. $$
    Furthermore, for each $k\geq 2$, $\ed_{\hcomp}(p)=g_k(p)$ for $p\in\bigl(p_{k},p_{k+1}\bigr]$ and $\ed_{\hcomp}(p)=1-p$ for $p\in\bigl[3/4,1\bigr]$.
    
    Consequently, $\dist\bigl(n,\hcomp\bigr)=\tfrac{5}{18}\tnct-o\bigl(n^2\bigr)$ and the maximum of $\ed_{\hcomp}(p)$ occurs at $p=p_5=2/3$.
    Moreover, if $\epsilon\in (0,0.001)$ and $\dist\bigl(G,\hcomp)\geq\bigl(\tfrac{5}{18}-\epsilon\bigr)\tnct$, then $\bigl|E(G)\bigr|/\tnct\in \bigl(\tfrac{2}{3}-\sqrt{2\epsilon},\tfrac{2}{3}+12\epsilon\bigr)$.
    \label{thm:comp}
\end{theorem}

\section{Techniques for computing the edit distance function}
\label{sec:edit}

The key objects in computing the edit distance function were called colored regularity graphs (CRGs) by Alon and Stav~\cite{AS}. 
\begin{definition}
A \emph{colored regularity graph (CRG)} $K$ is a simple complete graph together with a partition of its vertices into white and black vertices and a partition of its edges into white, black, and gray edges.
    We will use $K=\bigl(\VW,\VB;\EW,\EB,\EG\bigr)$ to denote that $K$ is a CRG with white vertex set $\VW$, black vertex set $\VB$, white edge set $\EW$, black edge set $\EB$ and gray edge set $\EG$. 
    
    As far as notation, the dependence on $K$ is usually suppressed if the CRG is understood. 
\end{definition}
Colored regularity graphs come from the analysis of graphs using Szemer\'edi's regularity lemma. 
In particular, one can apply the regularity lemma a second time, inside each of the clusters given by the first iteration and then Ramsey's theorem.
A convenient form of the regularity lemma for these purposes is given by Alon, Fischer, Krivelevich, and Szegedy~\cite{AFKSz}. 

CRGs are used to represent a graph after the editing is finished. 
As such, we need to demonstrate how forbidden graphs are not in the edited graph. 
This yields the idea of a colored homomorphism.
If $G$ is a graph, then $\overline{G}$ denotes the graph complement. 
The notation $\sqcup$ denotes the disjoint union of sets.

\begin{definition}
    A \emph{colored homomorphism} from a graph $G$ to a CRG $K=(\VW, \VB;$ $\EW, \EB, \EG)$ is a map $\varphi: V(G)\rightarrow V(K)$ such that 
    \begin{itemize}
        \item If $uv\in E(\overline{G})$, then either $\varphi(u)=\varphi(v)\in\VW$ or $\varphi(u)\neq\varphi(v)$ and $\varphi(u)\varphi(v)\in\EW\sqcup\EG$. 
        \item If $uv\in E(G)$, then either $\varphi(u)=\varphi(v)\in\VB$ or $\varphi(u)\neq\varphi(v)$ and $\varphi(u)\varphi(v)\in\EB\sqcup\EG$.
    \end{itemize}
    If there is a colored homomorphism from $G$ to $K$, then we write $G\mapsto K$. 
    If there is no colored homomorphism from $G$ to $K$, then we write $G\not\mapsto K$.
\end{definition}

Finally, we need to classify the colored regularity graphs that can represent graphs in $\HH$. 
\begin{definition}
    Given a hereditary property $\HH=\bigcap_{F\in\FF(\HH)}\forb(F)$, where $\FF(\HH)$ denotes the set of graphs that do not appear as induced subgraphs of members of $\HH$, the set of \emph{CRGs associated with $\HH$}, denoted $\KK(\HH)$, are the CRGs $K$, for which $F\not\mapsto K$ for all $F\in\FF(\HH)$. 
\end{definition}

For each CRG $K$, there are two functions, $f_K(p)$ and $g_K(p)$ that asymptotically measure the amount of editing that is done to $G(n,p)$ according to $K$.
In the case of $f_K(p)$, it is where the graph is partitioned into $|V(K)|$ equally-sized parts. 
In the case of $g_K(p)$, it is where the graph is partitioned optimally. 
\begin{definition}
    Let $K=\bigl(\VW,\VB;\EW,\EB,\EG\bigr)$ be a CRG on vertex set $\VW\cup\VB=\{v_1,\ldots,v_k\}$. 
    Let ${\bf M}_K(p)$ denote the matrix with entries defined as follows:
    $$
        m_{K}(p)_{ij} = \begin{cases}
                            p,      &   \text{if }i\neq j\text{ and }v_iv_j\in\EW\text{ or if }i=j\text{ and }v_i\in\VW; \\
                            0,      &   \text{if }i\neq j\text{ and }v_iv_j\in\EG;\\
                            1-p,    &   \text{if }i\neq j\text{ and }v_iv_j\in\EB\text{ or if }i=j\text{ and }v_i\in\VB.
                       \end{cases}
    $$
    The functions $f_K$ and $g_K$ are defined as follows:
    \begin{align*}
        f_K(p)  &=  \tfrac{1}{k^2} \bigl[p\bigl(|\VW|+2|\EW|\bigr) + (1-p)\bigl(|\VB|+2|\EB|\bigr)\bigr] = \tfrac{1}{k^2}{\bf 1}^T{\bf M}_K(p){\bf 1} \\
        g_K(p)  &=  \min\bigl\{\x^T\cdot {\bf M}_K(p)\cdot\x : \x^T{\bf 1}=1, \x\geq{\bf 0}\bigr\} .  
    \end{align*}
    The vector ${\bf 0}$ is the all-zeroes vector, ${\bf 1}$ is the all-ones vector and vector inequalities are coordinatewise. 
    \label{defn:fandg}
\end{definition}

In the theorem below, which shows that we only need to consider CRGs to compute the edit distance function, the ``$\inf$'' expressions are due to Balogh and the second author whereas the ``$\min$'' equality was established by Marchant and Thomason. 
\begin{theorem}[Balogh-M.~\cite{BM}; Marchant-Thomason~\cite{MT}]
    Let $\HH$ be a nontrivial hereditary property. For any $p\in [0,1]$,
    \begin{align*}
        \ed_{\HH}(p)    &=  \inf\bigl\{f_K(p) : K\in\KK(\HH)\bigr\} = \inf\bigl\{g_K(p) : K\in\KK(\HH)\bigr\} \\
                        &=  \min\bigl\{g_K(p) : K\in\KK(\HH)\bigr\} .
    \end{align*}
\end{theorem}

If we have a CRG $K$ and can delete a vertex without raising the value of $g_K(p)$, we would prefer to work with the smaller CRG. 
Thus, for any value of $p$, we show that we can restrict to a certain type of CRG which we call $p$-core.
\begin{definition}
    If $K$ is a CRG, then $K'$ is a \emph{sub-CRG of $K$} if it is formed by deleting vertices from $K$ and is \emph{proper} if $K'\neq K$.

    For any $p\in (0,1)$, a CRG $K$ is \emph{$p$-core} if for every proper sub-CRG $K'$, $g_{K'}(p)>g_{K}(p)$. 

    For any $p\in (0,1)$ and nontrivial hereditary property $\HH$, we denote $\KK_p(\HH)$ to be the set of all CRGs in $\KK(\HH)$ that are also $p$-core. 
\end{definition}

Marchant and Thomason gave a strong characterization for $p$-core CRGs. 
\begin{theorem}[Marchant-Thomason~\cite{MT}]
    Let $K$ be a $p$-core CRG with $V(G)=\VW\sqcup\VB$ and $E(K)=\EW\sqcup\EG\sqcup\EB$ and let $v,w\in V(K)$.
    \begin{enumerate}[(a)]
        \item If $v\in\VW$ and $w\in\VB$, then $vw\in\EG$. \label{it:MT:vwvb}
        \item If $p\leq 1/2$ and $v,w\in\VW$, then $vw\in\EG$. \label{it:MT:psmall:vw}
        \item If $p\leq 1/2$ and $v,w\in\VB$, then $vw\in\EW\cup\EG$. \label{it:MT:psmall:vb}
        \item If $p\geq 1/2$ and $v,w\in\VB$, then $vw\in\EG$. \label{it:MT:plarge:vb}
        \item If $p\geq 1/2$ and $v,w\in\VW$, then $vw\in\EB\cup\EG$. \label{it:MT:plarge:vw}
    \end{enumerate}
    \label{thm:MT}
\end{theorem}

A \emph{component} of a CRG is a connected component in the graph induced by non-gray edges. 
Riasanovsky and the second author~\cite{MR} gave a full characterization of $p$-core CRGs for all $p\in \bigl[0.381967,0.618033\bigr]$.
\begin{theorem}[M.-Riasanovsky~\cite{MR}]
    For $p\in \bigl(\frac{3-\sqrt{5}}{2},\frac{\sqrt{5}-1}{2}\bigr)$, a $p$-core CRG has the property that all of its components are cliques.
\end{theorem}

In fact, $g_K$ can be computed from its components.
\begin{theorem}[M.~\cite{Mar}]
    Let $K$ be a CRG with components $K_{(1)},K_{(2)},\ldots,K_{(\ell)}$. Then,
    $$ \Bigl(g_K(p)\Bigr)^{-1} = \sum_{i=1}^{\ell} \Bigl(g_{K_{(i)}}(p)\Bigr)^{-1} . $$
    \label{thm:components}
\end{theorem}

If a CRG has $V(G)=\VW\sqcup\VB$ and $E(G)=\EG$ with $|\VW|=r$ and $|\VB|=s$, then we denote it as $K(r,s)$. 
\begin{corollary}
    For all nonnegative integers $r$ and $s$ such that $r+s\geq 1$, we have $g_{K(r,s)}(p)=\frac{p(1-p)}{r(1-p)+sp}$.
    \label{cor:components}
\end{corollary}

Finally, we establish the main tool with which we use to compute $g_K$ for a CRG $K$. 
Recall that Definition~\ref{defn:fandg} establishes an optimal weight vector $\x$ such that $g_K(p)=\x^T {\bf M}_K(p) \x$.
If $v\in V(G)$, then we denote $\dg(v)=\sum_{vw\in\EG(K)} \x(w)$ to be the \emph{weighted gray neighborhood of $v$}.
\begin{theorem}[M.~\cite{Mar}]
    Let $p\in (0,1)$ and let $K=\bigl(\VW,\VB;\EW,\EB,\EG\bigr)$ be a $p$-core CRG with optimum weight vector $\x$.
    Then this vector is unique and it obeys the matrix equation
    $$ {\bf M}_K(p)\cdot\x = g_K(p){\bf 1} .$$
    Furthermore,
    \begin{enumerate}[(a)]
        \item   If $p\leq 1/2$, then $\x(v)=g_K(p)/p$ for all $v\in\VW$ and
        $$ \dg(v)=\frac{p-g_K(p)}{p} + \frac{1-2p}{p} \x(v), \text{ for all } v\in\VB.$$
        \item   If $p\geq 1/2$, then $\x(v)=g_K(p)/(1-p)$ for all $v\in\VB$ and
        $$ \dg(v)=\frac{1-p-g_K(p)}{1-p} + \frac{2p-1}{1-p} \x(v), \text{ for all } v\in\VW.$$ \label{it:graydeg:plarge}
    \end{enumerate}  
    \label{thm:graydeg}
\end{theorem}

\begin{corollary}
    Let $k\geq 2$ be an integer.
    \begin{enumerate}[(a)]
        \item Let $p<1/2$ and let $K$ be a $p$-core CRG on $k$ black vertices.
        If $K$ has no gray edges, then $g_K(p)=p+\frac{1-2p}{k}$. \label{it:cor:graydeg:psmall}
        \item Let $p>1/2$ and let $K$ be a $p$-core CRG on $k$ white vertices.
        If $K$ has no gray edges, then $g_K(p)=1-p+\frac{2p-1}{k}$. \label{it:cor:graydeg:plarge}
    \end{enumerate}
    \label{cor:graydeg}
\end{corollary}

For this paper, we will address CRGs for which the gray edges form a spanning star.
\begin{proposition}
    Let $t\geq 1$ be an integer.
    \begin{enumerate}[(a)]
        \item Let $p<1/2$ and let $K$ be a $p$-core CRG on $t+1$ black vertices.
        If the gray edges of $K$ induce a $K_{1,t}$, then $g_K(p)=p(1-p)+\frac{(1-2p)(1-p)^2}{t+1-2p}$.
        \item Let $p>1/2$ and let $K$ be a $p$-core CRG on $t+1$ white vertices.
        If the gray edges of $K$ induce a $K_{1,t}$, then $g_K(p)=p(1-p)+\frac{(2p-1)p^2}{t+2p-1}$. \label{it:star:plarge}
    \end{enumerate}
    \label{prop:star}
\end{proposition}

\begin{proof}
    Because we will use \eqref{it:star:plarge} later in the paper, we will prove that part. 
    For the other part, we make the observation that if CRG $\overline{K}$ is the photographic negative of $K$ (i.e., black vertices and edges are made white, and white vertices and edges are made black) then $g_{\overline{K}}(p)=g_K(1-p)$.

    In order to prove \eqref{it:star:plarge}, for convenience let $g=g_K(p)$, $G=\frac{1-p-g}{1-p}$, and $P=\frac{2p-1}{1-p}$. 
    If $K$ is $p$-core, and $c$ denotes the weight of the center of the star, then $\dg(w)=c$ for each leaf $w$ and so $\x(w)=(1-c)/t$. As a result, the following hold from Theorem~\ref{thm:graydeg}\eqref{it:graydeg:plarge}: $G+Pc=1-c$ and $G+P\frac{1-c}{t}=c$. 
    Thus,
    \begin{align*}
        c   =   \frac{1-G}{1+P}     &=  \frac{tG+P}{t+P} \\
                G                   &=  \frac{t-P^2}{tP+2t+P} \\
                \frac{1-p-g}{1-p}   &=  \frac{t(1-p)^2-(2p-1)^2}{(1-p)\bigl(t(2p-1)+2t(1-p)+(2p-1)\bigr)} \\
                g                   &=  \frac{tp(1-p)+p(2p-1)}{t+2p-1} = p(1-p)+\frac{(2p-1)p^2}{t+2p-1} .
    \end{align*}
    Now we have to verify that $K$ is $p$-core for $p>1/2$ by showing that $g_{K'}(p)>g_K(p)$ for all sub-CRGs $K'$. 
    Since the expression increases as $t$ decreases, deleting leaves will result in a sub-CRG that has a larger $g$ function. 
    If, on the other hand, $K'$ comes from deleting the center as well as $t-t'\geq 0$ leaves, then by Corollary~\ref{cor:graydeg}, the resulting $g_{K'}(p)$ will be $1-p+\frac{2p-1}{t'}$.
    However,
    \begin{align*} 
        g_K(p)-g_{K'}(p)    
        &=  \biggl[p(1-p)+\frac{(2p-1)p^2}{t+2p-1}\biggr] - \biggl[1-p+\frac{2p-1}{t'}\biggr] \\
        &=  -(1-p)^2+(2p-1)\biggl[\frac{p^2}{t+2p-1}-\frac{1}{t'}\biggr] \\
        &\leq  -(1-p)^2+(2p-1)\biggl[\frac{p^2}{t+2p-1}-\frac{1}{t}\biggr] \\
        &=  \frac{-(1-p)^2t(t+2p-1)+(2p-1)t(1-p^2)-(2p-1)^2}{t(t+2p-1)} \\
        &=  \frac{-(1-p)^2t^2-2p(1-p)(2p-1)t-(2p-1)^2}{t(t+2p-1)} < 0.
    \end{align*}
    And this establishes that $K$ is $p$-core for $p>1/2$. 
\end{proof}

\section{Proof of Theorem~\ref{thm:word}}
\label{sec:proof:word}

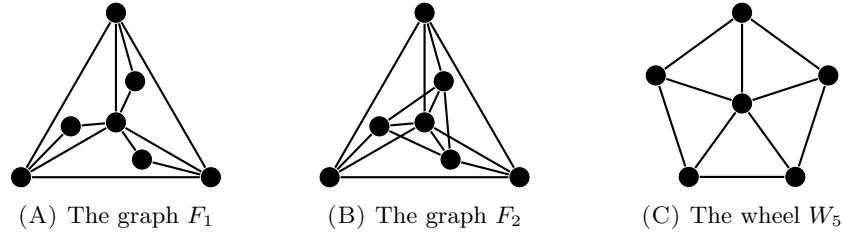
\begin{figure}
    \centering
    \ \hfill \
    \begin{subfigure}[b]{0.30\textwidth}
        \centering
        \begin{tikzpicture}[scale=1]
            \centering
            \useasboundingbox (-1.6,-1.25) rectangle (1.6,1.25);
            \coordinate (offset) at (0,-0.4);
            \coordinate (c) at ($(0,0)+(offset)$);
            \foreach \i in {1,2,3}{
                \coordinate (v\i) at ($(-30+120*\i:1.45)+(offset)$);
                \coordinate (w\i) at ($(-55+120*\i:0.60)+(offset)$);
            }
            \begin{pgfonlayer}{fore}
                \draw (c) node[vtx] {};
                \foreach \i in {1,2,3}{
                    \draw (v\i) node[vtx] {};
                    \draw (w\i) node[vtx] {};
                }
            \end{pgfonlayer}
            \begin{pgfonlayer}{main}
                \foreach \i in {1,2}{
                    \draw[edg] (c) -- (v\i);
                    \draw[edg] (c) -- (w\i);
                    \draw[edg] (v\i) -- (w\i);
                    \pgfmathtruncatemacro{\j}{\i+1};
                    \draw[edg] (v\i) -- (v\j);
                }
                \draw[edg] (c) -- (v3);
                \draw[edg] (c) -- (w3);
                \draw[edg] (v3) -- (w3);
                \draw[edg] (v1) -- (v3);
            \end{pgfonlayer}
        \end{tikzpicture}
        \caption{The graph $F_1$}
        \label{fig:F1-aaa}
    \end{subfigure}\ \hfill \
    \begin{subfigure}[b]{0.30\textwidth}
        \centering
        \begin{tikzpicture}[scale=1]
            \centering
            \useasboundingbox (-1.6,-1.25) rectangle (1.6,1.25);
            \coordinate (offset) at (0,-0.4);
            \coordinate (c) at ($(0,0)+(offset)$);
            \foreach \i in {1,2,3}{
                \coordinate (v\i) at ($(-30+120*\i:1.45)+(offset)$);
                \coordinate (w\i) at ($(-55+120*\i:0.60)+(offset)$);
            }
            \begin{pgfonlayer}{fore}
                \draw (c) node[vtx] {};
                \foreach \i in {1,2,3}{
                    \draw (v\i) node[vtx] {};
                    \draw (w\i) node[vtx] {};
                }
            \end{pgfonlayer}
            \begin{pgfonlayer}{main}
                \foreach \i in {1,2}{
                    \draw[edg] (c) -- (v\i);
                    \draw[edg] (c) -- (w\i);
                    \draw[edg] (v\i) -- (w\i);
                    \pgfmathtruncatemacro{\j}{\i+1};
                    \draw[edg] (v\i) -- (v\j);
                    \draw[edg] (w\i) -- (w\j);
                }
                \draw[edg] (c) -- (v3);
                \draw[edg] (c) -- (w3);
                \draw[edg] (v3) -- (w3);
                \draw[edg] (v1) -- (v3);
                \draw[edg] (w1) -- (w3);
            \end{pgfonlayer}
        \end{tikzpicture}
        \caption{The graph $F_2$}
        \label{fig:F2}
    \end{subfigure} \ \hfill \
    \begin{subfigure}[b]{0.30\textwidth}
        \centering
        \begin{tikzpicture}[scale=1]
            \centering
            \useasboundingbox (-1.6,-1.25) rectangle (1.6,1.25);
            \coordinate (offset) at (0,-0.15);
            \coordinate (c) at ($(0,0)+(offset)$);
            \foreach \i in {1,2,3,4,5}{
                \coordinate (v\i) at ($(18+72*\i:1.20)+(offset)$);
            }
            \begin{pgfonlayer}{fore}
                \draw (c) node[vtx] {};
                \foreach \i in {1,2,3,4,5}{
                    \draw (v\i) node[vtx] {};
                }
            \end{pgfonlayer}
            \begin{pgfonlayer}{main}
                \foreach \i in {1,2,3,4}{
                    \draw[edg] (c) -- (v\i);
                    \pgfmathtruncatemacro{\j}{\i+1};
                    \draw[edg] (v\i) -- (v\j);
                }
                \draw[edg] (c) -- (v5);
                \draw[edg] (v1) -- (v5);
            \end{pgfonlayer}
        \end{tikzpicture} 
        \caption{The wheel $W_5$} 
    \end{subfigure} \ \hfill \
	\caption{Examples of non-word-representable graphs.}
    \label{fig:forb:word}
\end{figure}
        
\begin{theorem}[Halld\'orsson-K.-Pyatkin~\cite{HKP}]
    If $G$ is $3$-colorable or $G$ is a clique, then $G$ is word-representable, that is, $G\in\hword$. 
    \label{thm:HKP}
\end{theorem}

From Theorem~\ref{thm:HKP}, we immediately get the upper bound for $\dist\bigl(n,\hword\bigr)$.
\begin{lemma}
    For any graph $G$,
    \begin{align*}
        \dist\bigl(G,\hword\bigr)\leq\min\Bigl\{|E(G)|/3,\tnct-|E(G)|\Bigr\} .
    \end{align*}
    Hence, $\dist\bigl(n,\hword\bigr)\leq \frac{1}{4}\tnct$.
\end{lemma}

In addition if $\dist\bigl(G,\hkword{k}\bigr)\geq\bigl(\frac{1}{4}-\epsilon\bigr)\tnct$, then $\min\bigl\{|E(G)|/3,\tnct-|E(G)|\bigr\}\geq\bigl(\frac{1}{4}-\epsilon\bigr)\tnct$ and so $\bigl|E(G)\bigr|/\tnct\in \bigl[\tfrac{3}{4}-3\epsilon,\tfrac{3}{4}+\epsilon\bigr]$.

\begin{proof}
    To prove the upper bound for Theorem~\ref{thm:kword}, consider two possible operations on the edge set of $G$:
    \begin{enumerate}[(1)]
        \item Color each of the vertices of $G$ in three distinct colors. 
        If the endvertices of an edge have the same color, delete the edge. 
        Otherwise, do nothing. \label{it:recipe1}
        \item Add every nonedge to make a complete graph. \label{it:recipe2}
    \end{enumerate}

    In case~\eqref{it:recipe1}, we can color each of the vertices independently and uniformly at random with equal probability $1/3$. 
    The expected number of deleted edges is $|E(G)|/3$. 
    In case~\eqref{it:recipe2}, the number of added edges is $\binom{n}{2}-|E(G)|$. 
    Hence, there is a way of editing $G$ with at most $\min\bigl\{|E(G)|/3,\tnct-|E(G)|\bigr\}$ changed edges.
\end{proof}

As far as the lower bound for Theorem~\ref{thm:kword}, it is our goal to show that $\ed_{\hword}(p)\geq\min\bigl\{p/3,1-p\bigr\}$.

Although there are many minimal non-word-representable graphs, we will only require $3$ of them, as established by Pyatkin and the first author.
\begin{theorem}[K.-Pyatkin~\cite{KP}]
    The graphs $F_1$, $F_2$, and $W_5$ (the wheel on $6$ vertices) which are depicted in Figure~\ref{fig:forb:word} are not word-representable.
    Hence, $\hword\subset\forb\bigl\{F_1,F_2,W_5\bigr\}$.
    \label{thm:KP}
\end{theorem}

With that information, we have Proposition~\ref{prop:word:structure}. 
Recall that $K(r,s)$ denotes the CRG with $r$ white vertices and $s$ black vertices and all edges gray.
\begin{proposition}
    The following embeddings exist:
    \begin{enumerate}[(a)]
        \item $F_1\mapsto K(1,1)$. \label{it:word:structure:k11}
        \item $F_2\mapsto K(0,2)$. \label{it:word:structure:k02}
        \item $W_5\mapsto K(4,0)$. \label{it:word:structure:k40}
        \item If $K$ has at least $5$ vertices all of which are white, where each edge is either black or gray and such that the gray edges have a $K_{1,3}$ as a subgraph, then $F_1\mapsto K$. \label{it:word:structure:claw}
        \item If $t\in\{3,4,5\}$ and $K$ has at least $t+1$ vertices all of which are white, where each edge is either black or gray and such that the gray edges have a $C_t$ as a subgraph, then $W_5\mapsto K$. \label{it:word:structure:cycle}
    \end{enumerate}
    \label{prop:word:structure}
\end{proposition}

\begin{proof}
    For \eqref{it:word:structure:k11}, we observe that $F_1$ is a split graph. 
    That is, $F_1$ can be partitioned into two parts, one is an independent set and one is a clique.
    The independent set maps to the white vertex and the clique maps to the black vertex. 

    For \eqref{it:word:structure:k02}, we observe that $\chi\bigl(\overline{F_2}\bigr)=2$.
    Hence, $F_2$ can be partitioned into two parts, each of which is a clique.
    Each of the two cliques maps to a different black vertex.

    For \eqref{it:word:structure:k40}, we observe that $\chi\bigl(W_5)=4$. 
    Hence, $W_5$ can be partitioned into four parts, each of which is an independent set.
    Each of the four independent sets maps to a different white vertex.

    For \eqref{it:word:structure:claw}, we demonstrate the mapping in Figure~\ref{fig:F1toK13}.
    The graph $F_1$ is on $7$ vertices and has a degree-$6$ vertex. 
    The vertices of the independent set of size $3$, consisting of the degree-$2$ vertices, are all mapped to the center of the gray $K_{1,3}$. 
    The remaining degree-$4$ vertices are mapped to different leaves of the gray $K_{1,3}$. 
    The degree-$6$ vertex can be mapped to any additional fifth CRG-vertex because it is adjacent to every other vertex and the edges from the fifth CRG-vertex can be either gray or black.
    
    \begin{figure}
        \centering
        \ \hfill \
        \begin{subfigure}[b]{0.30\textwidth}
            \centering
            \begin{tikzpicture}[scale=1]
                \centering
                \useasboundingbox (-1.6,-1.25) rectangle (1.6,1.25);
                \coordinate (offset) at (0,-0.4);
                \coordinate (c) at ($(0,0)+(offset)$);
                \foreach \i in {1,2,3}{
                    \coordinate (v\i) at ($(-30+120*\i:1.45)+(offset)$);
                    \coordinate (w\i) at ($(-55+120*\i:0.60)+(offset)$);
                }
                \begin{pgfonlayer}{fore}
                    \draw (c) node[vtx,ultra thick,black,fill=white] {};
                    \draw (v1) node[vtx,label={[label distance=-1pt]{0}:$2$}] {};
                    \draw (v2) node[vtx,label={[label distance=-2pt]{95}:$3$}] {};
                    \draw (v3) node[vtx,label={[label distance=-2pt]{85}:$4$}] {};
                    \draw (w1) node[vtx,label={[label distance=-3pt]{-60}:$1$}] {};
                    \draw (w2) node[vtx,label={[label distance=-3pt]{85}:$1$}] {};
                    \draw (w3) node[vtx,label={[label distance=-3pt]{180}:$1$}] {};
                \end{pgfonlayer}
                \begin{pgfonlayer}{main}
                    \foreach \i in {1,2}{
                        \draw[edg] (c) -- (v\i);
                        \draw[edg] (c) -- (w\i);
                        \draw[edg] (v\i) -- (w\i);
                        \pgfmathtruncatemacro{\j}{\i+1};
                        \draw[edg] (v\i) -- (v\j);
                    }
                    \draw[edg] (c) -- (v3);
                    \draw[edg] (c) -- (w3);
                    \draw[edg] (v3) -- (w3);
                    \draw[edg] (v1) -- (v3);
                \end{pgfonlayer}
            \end{tikzpicture}
            \caption*{The graph $F_1$}
            \label{fig:F1}
        \end{subfigure}\ \hfill \
        \begin{subfigure}[b]{0.65\textwidth}
            \centering
            \begin{tikzpicture}[scale=1]
                \centering
                \useasboundingbox (-1.6,-1.25) rectangle (1.6,1.25);
                \coordinate (offset) at (0,-0.4);
                \coordinate (c) at ($(0,-0.50)+(offset)$);
                \coordinate (v2) at ($(-1.25,1.30)+(offset)$);
                \coordinate (v3) at ($(0,0.65)+(offset)$);
                \coordinate (v4) at ($(1.25,1.30)+(offset)$);
                \begin{pgfonlayer}{forefore}
                    \draw (c) node {$1$};
                    \draw (v2) node {$2$};
                    \draw (v3) node {$3$};
                    \draw (v4) node {$4$};
                \end{pgfonlayer}
                \begin{pgfonlayer}{fore}
                    \draw[thick, black, fill=white] (c) circle (10pt);
                    \foreach \i in {2,3,4}{
                        \draw[thick, black, fill=white] (v\i) circle (10pt);
                    }
                \end{pgfonlayer}
                \def\wdt{0.09};
                \begin{pgfonlayer}{main}
                    \draw[ultra thin, white, fill=gray!40] ($(c)+(-0.88*\wdt,-0.44*\wdt)$) -- ($(c)+(0.88*\wdt,0.44*\wdt)$) -- ($(v2)+(0.88*\wdt,0.44*\wdt)$) -- ($(v2)+(-0.88*\wdt,-0.44*\wdt)$) -- cycle;
                    \draw[ultra thin, white, fill=gray!40] ($(c)+(-0.88*\wdt,0.44*\wdt)$) -- ($(c)+(0.88*\wdt,-0.44*\wdt)$) -- ($(v4)+(0.88*\wdt,-0.44*\wdt)$) -- ($(v4)+(-0.88*\wdt,0.44*\wdt)$) -- cycle;                
                    \draw[ultra thin, white, fill=gray!40] ($(c)+(-\wdt,0)$) -- ($(c)+(\wdt,0)$) -- ($(v3)+(\wdt,0)$) -- ($(v3)+(-\wdt,0)$) -- cycle;
                \end{pgfonlayer}
                \begin{pgfonlayer}{back}
                    \draw[ultra thin, white, fill=black] ($(v3)+(-0.44*\wdt,-0.88*\wdt)$) -- ($(v3)+(0.44*\wdt,0.88*\wdt)$) -- ($(v2)+(0.44*\wdt,0.88*\wdt)$) -- ($(v2)+(-0.44*\wdt,-0.88*\wdt)$) -- cycle;
                    \draw[ultra thin, white, fill=black] ($(v3)+(-0.44*\wdt,0.88*\wdt)$) -- ($(v3)+(0.44*\wdt,-0.88*\wdt)$) -- ($(v4)+(0.44*\wdt,-0.88*\wdt)$) -- ($(v4)+(-0.44*\wdt,0.88*\wdt)$) -- cycle;
                    \draw[ultra thin, white, fill=black] ($(v2)+(0,-\wdt)$) -- ($(v2)+(0,\wdt)$) -- ($(v4)+(0,\wdt)$) -- ($(v4)+(0,-\wdt)$) -- cycle;
                \end{pgfonlayer}
            \end{tikzpicture}
            \caption*{A $p$-core CRG, $p>1/2$ with a gray $K_{1,3}$}
            \label{fig:F2-aa}
        \end{subfigure} \ \hfill \
    	\caption{The mapping from $F_1$ to a CRG with a gray $K_{1,3}$. We may assume that all other edges of the CRG are black, otherwise the embedding is less restrictive. The central degree-$6$ vertex of $F_1$ can be placed in any additional vertex.}
        \label{fig:F1toK13}
    \end{figure}
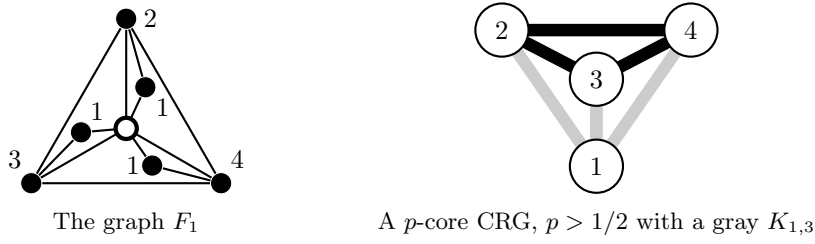    

    Finally, for \eqref{it:word:structure:cycle}, we demonstrate the mapping in Figure~\ref{fig:W5toC3C4} in the cases of $t\in\{3,4\}$.
    The graph $W_5$ is on $6$ vertices and has a degree-$5$ vertex. 
    The vertices of the complement of the cycle are mapped as independent sets to the vertices of $C_t$. 
    The degree-$5$ vertex can be mapped to any additional CRG-vertex because it is adjacent to every other vertex and the edges from the CRG-vertex can be either gray or black.
\end{proof}

    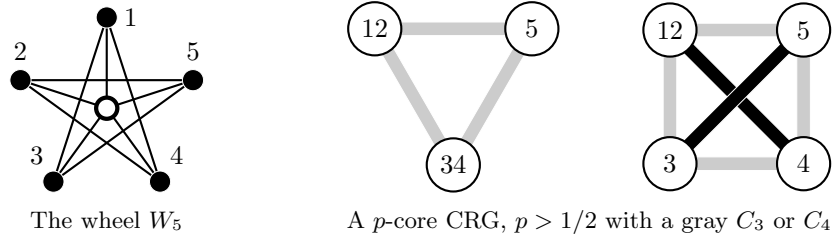
\begin{figure}
        \centering
        \ \hfill \
        \begin{subfigure}[b]{0.30\textwidth}
            \centering
            \begin{tikzpicture}[scale=1]
                \centering
                \useasboundingbox (-1.6,-1.25) rectangle (1.6,1.25);
                \coordinate (offset) at (0,-0.15);
                \coordinate (c) at ($(0,0)+(offset)$);
                \foreach \i in {1,2,3,4,5}{
                    \coordinate (v\i) at ($(18+72*\i:1.20)+(offset)$);
                }
                \begin{pgfonlayer}{fore}
                    \draw (c) node[vtx,ultra thick,black,fill=white] {};
                    \draw (v1) node[vtx,label={[label distance=-1pt]{0}:$1$}] {};
                    \draw (v2) node[vtx,label={[label distance=0pt]{90}:$2$}] {};
                    \draw (v3) node[vtx,label={[label distance=-1pt]{95}:$3$}] {};
                    \draw (v4) node[vtx,label={[label distance=-1pt]{85}:$4$}] {};
                    \draw (v5) node[vtx,label={[label distance=0pt]{90}:$5$}] {};
            \end{pgfonlayer}
            \begin{pgfonlayer}{main}
                \foreach \i in {1,2,3,4,5}{
                    \draw[edg] (c) -- (v\i);
                }
                \draw[edg] (v1) -- (v3);
                \draw[edg] (v2) -- (v4);
                \draw[edg] (v3) -- (v5);
                \draw[edg] (v4) -- (v1);
                \draw[edg] (v5) -- (v2);
            \end{pgfonlayer}
        \end{tikzpicture} 
        \caption*{The wheel $W_5$} 
    \end{subfigure} \ \hfill \        
        \begin{subfigure}[b]{0.65\textwidth}
            \centering
            \begin{tikzpicture}[scale=1]
                \centering
                \useasboundingbox (-3.2,-1.25) rectangle (3.2,1.25);
                \coordinate (c3offset) at (-1.8,+0.3);
                \coordinate (c4offset) at (+1.95,+0.0);
                \foreach \i in {1,2,3}{
                    \coordinate (v\i) at ($(30+120*\i:1.2)+(c3offset)$);
                }
                \foreach \j in {1,2,3,4}{
                    \coordinate (w\j) at ($(45+90*\j:1.25)+(c4offset)$);
                }                
                \begin{pgfonlayer}{forefore}
                    \draw (v1) node {$12$};
                    \draw (v2) node {$34$};
                    \draw (v3) node {$5$};
                    \draw (w1) node {$12$};
                    \draw (w2) node {$3$};
                    \draw (w3) node {$4$};
                    \draw (w4) node {$5$};
                \end{pgfonlayer}
                \begin{pgfonlayer}{fore}
                    \foreach \i in {1,2,3}{
                        \draw[thick, black, fill=white] (v\i) circle (10pt);
                    }
                    \foreach \j in {1,2,3,4}{
                        \draw[thick, black, fill=white] (w\j) circle (10pt);
                    }
                \end{pgfonlayer}
                \def\wdt{0.09};
                \begin{pgfonlayer}{main}
                    \draw[ultra thin, white, fill=gray!40] ($(v1)+(-0.88*\wdt,-0.44*\wdt)$) -- ($(v1)+(0.88*\wdt,0.44*\wdt)$) -- ($(v2)+(0.88*\wdt,0.44*\wdt)$) -- ($(v2)+(-0.88*\wdt,-0.44*\wdt)$) -- cycle;
                    \draw[ultra thin, white, fill=gray!40] ($(v2)+(-0.88*\wdt,0.44*\wdt)$) -- ($(v2)+(0.88*\wdt,-0.44*\wdt)$) -- ($(v3)+(0.88*\wdt,-0.44*\wdt)$) -- ($(v3)+(-0.88*\wdt,0.44*\wdt)$) -- cycle;                
                    \draw[ultra thin, white, fill=gray!40] ($(v3)+(0,-\wdt)$) -- ($(v3)+(0,\wdt)$) -- ($(v1)+(0,\wdt)$) -- ($(v1)+(0,-\wdt)$) -- cycle;
                    \draw[ultra thin, white, fill=gray!40] ($(w1)+(-\wdt,0)$) -- ($(w1)+(\wdt,0)$) -- ($(w2)+(\wdt,0)$) -- ($(w2)+(-\wdt,0)$) -- cycle;
                    \draw[ultra thin, white, fill=gray!40] ($(w2)+(0,-\wdt)$) -- ($(w2)+(0,\wdt)$) -- ($(w3)+(0,\wdt)$) -- ($(w3)+(0,-\wdt)$) -- cycle;
                    \draw[ultra thin, white, fill=gray!40] ($(w3)+(-\wdt,0)$) -- ($(w3)+(\wdt,0)$) -- ($(w4)+(\wdt,0)$) -- ($(w4)+(-\wdt,0)$) -- cycle;
                    \draw[ultra thin, white, fill=gray!40] ($(w4)+(0,-\wdt)$) -- ($(w4)+(0,\wdt)$) -- ($(w1)+(0,\wdt)$) -- ($(w1)+(0,-\wdt)$) -- cycle;
                \end{pgfonlayer}
                \begin{pgfonlayer}{back}
                    \draw[ultra thin, white, fill=black] ($(w1)+(-0.71*\wdt,-0.71*\wdt)$) -- ($(w1)+(0.71*\wdt,0.71*\wdt)$) -- ($(w3)+(0.71*\wdt,0.71*\wdt)$) -- ($(w3)+(-0.71*\wdt,-0.71*\wdt)$) -- cycle;
                    \draw[ultra thin, white, fill=black] ($(w2)+(0.71*\wdt,-0.71*\wdt)$) -- ($(w2)+(-0.71*\wdt,0.71*\wdt)$) -- ($(w4)+(-0.71*\wdt,0.71*\wdt)$) -- ($(w4)+(0.71*\wdt,-0.71*\wdt)$) -- cycle;
                \end{pgfonlayer}
            \end{tikzpicture}
            \caption*{A $p$-core CRG, $p>1/2$ with a gray $C_3$ or $C_4$}
            \label{fig:F2-aaa}
        \end{subfigure} \ \hfill \
    	\caption{The mapping from $W_5$ to a CRG with a gray $C_t$, $t\in\{3,4,5\}$. We may assume that all other edges of the CRG are black, otherwise the embedding is less restrictive. The central degree-$5$ vertex of $W_5$ can be placed in any additional vertex.}
        \label{fig:W5toC3C4}
    \end{figure} 

Using Proposition~\ref{prop:word:structure}, we can now severely restrict the class of CRGs under consideration.
Suppose that $K\in\KK_p(\HH)$ for some value of $p$. 
Since $F_1\mapsto K(1,1)$, Theorem~\ref{thm:MT}\eqref{it:MT:vwvb} gives that either $V(K)=\VW(K)$ or $V(K)=\VB(K)$. 

If $V(K)=\VB(K)$ then since $F_2\mapsto K(0,2)$, there can be no gray edge.
Thus, either $p\geq 1/2$ and Theorem~\ref{thm:MT}\eqref{it:MT:plarge:vb} gives that $|\VB(K)|=1$, hence $g_K(p)=1-p$, or $p<1/2$ and Theorem~\ref{thm:MT}\eqref{it:MT:psmall:vb} gives that all edges are black and Corollary~\ref{cor:graydeg} gives that $g_K(p)>p\geq p/3$.
Hence, the only black-vertex $p$-core CRG consists of a single vertex.

If $V(K)=\VW(K)$ and $p\leq 1/2$, then since $W_5\mapsto K(4,0)$, Theorem~\ref{thm:MT}\eqref{it:MT:psmall:vw} gives that $|\VW(K)|\leq 3$ and Corollary~\ref{cor:components} gives that $g_K(p)\geq p/3$.

Thus, we may assume that $V(K)=\VW(K)$ and $p>1/2$ because in every other case, $g_K(p)\geq\min\bigl\{p/3,1-p\bigr\}$.
By Theorem~\ref{thm:MT}\eqref{it:MT:plarge:vw}, all edges of $K$ are either gray or black.
From Proposition~\ref{prop:word:structure}, we see that the graph induced by the gray edges of $K$ satisfies one of the following:
\begin{enumerate}[(i)]
    \item $\EG(K)$ induces $K_{1,3}$. \label{it:case:claw}
    \item $\EG(K)$ induces $C_3$. \label{it:case:C3}
    \item $\EG(K)$ induces $C_4$. \label{it:case:C4}
    \item $\EG(K)$ induces $C_5$. \label{it:case:C5}
    \item $\EG(K)$ induces a graph with maximum degree at most 2 and girth at least 6. \label{it:case:last}
\end{enumerate}

For case~\eqref{it:case:claw}, Proposition~\ref{prop:star}\eqref{it:star:plarge} gives that if $p>1/2$, then $g_K(p) = p(1-p)+\frac{(2p-1)p^2}{3+2p-1} = \frac{p(2-p)}{2(1+p)}$. 
Since $\frac{p(2-p)}{2(1+p)}-\min\bigl\{p/3,1-p\bigr\} = \max\bigl\{\frac{p(4-5p)}{6(1+p)},\frac{-2+2p+p^2}{2(1+p)}\bigr\} \geq 1/56$, in this case $g_K(p) > \min\bigl\{p/3,1-p\bigr\}$.

For case~\eqref{it:case:C3}, Corollary~\ref{cor:components} gives that $g_K(p)=p/3$.

For case~\eqref{it:case:C4}, Corollary~\ref{cor:graydeg}\eqref{it:cor:graydeg:plarge} and Theorem~\ref{thm:components} give that if $p>1/2$, then $g_K(p)=1/4$.
Since $1/4-\min\bigl\{p/3,1-p\bigr\}=\max\bigl\{\frac{3-4p}{12},\frac{4p-3}{4}\bigr\}\geq 0$ (with equality if and only if $p=3/4$), in this case $g_K(p)\geq\min\bigl\{p/3,1-p\bigr\}$.

For case~\eqref{it:case:C5}, we consider the subgraphs $K'$ of $C_5$. 
If $K'$ has at most two vertices, then Corollary~\ref{cor:components} gives that $g_{K'}(p)\geq p/2$.
If $K'$ has three vertices, then the gray edges induce a subgraph of $P_3$, the path on 3 vertices and since changing a gray edge to black cannot increase the value of the $g$ function, we may assume that the gray edges induce $P_3$.
Thus, Proposition~\ref{prop:star}\eqref{it:star:plarge} gives that if $p>1/2$, then $g_{K'}(p)\geq p(1-p)+\frac{(2p-1)p^2}{2+2p-1}=\frac{p}{2p+1}\geq\frac{p}{3}$. 
If $K'$ has four vertices, then the gray edges induce a subgraph of $P_4$, the path on 4 vertices, and again we may assume that the gray edges induce $P_4$ itself.
We will see the case of the gray edges inducing $P_4$ more thoroughly in Section~\ref{sec:proof:comp}, but it is not difficult to see that if $p>1/2$, then $g_{K'}(p)=\min\bigl\{\frac{p}{1+2p},\frac{-1+3p-p^2}{2(3p-1)}\bigr\}$.
Since $\min\bigl\{\frac{p}{1+2p},\frac{-1+3p-p^2}{2(3p-1)}\bigr\}-\min\bigl\{p/3,1-p\bigr\}\geq 1/40$, in this case $g_K(p)>\min\bigl\{p/3,1-p\bigr\}$.
For $K$ itself, the gray edges induce $C_5$ and if $p>1/2$, then $g_{K}(p)=\min\bigl\{\frac{p}{1+2p},\frac{-1+3p-p^2}{2(3p-1)},\frac{2-p}{5}\bigr\}=\min\bigl\{\frac{p}{1+2p},\frac{2-p}{5}\bigr\}$.
Since $\min\bigl\{\frac{p}{1+2p},\frac{2-p}{5}\bigr\}-\min\bigl\{p/3,1-p\bigr\}\geq 0$ (with equality if and only if $p=3/4$), then case~\eqref{it:case:C5} $g_K(p)\geq\min\bigl\{p/3,1-p\bigr\}$.

Finally, we are left with case~\eqref{it:case:last}.
Consider the gray edge $v_1v_2$ such that $\x(v_1)+\x(v_2)$ is maximum. 
For convenience, let $g=g_K(p)$, $G=\frac{1-p-g}{1-p}$, $P=\frac{2p-1}{1-p}$, $x=\x(v_1)$, and $y=\x(v_2)$.

If $v_1$ and $v_2$ have no other neighbors in the gray-edge graph, then $G+Px=y$ and $G+Py=x$. 
Thus, $2G=(1-P)(x+y)$.
Hence $P<1$ (equivalently, $p<2/3$), $x+y\leq 1$ and 

\vspace{-2cm}

\begin{align*}
    2G = (1-P)(x+y)         &\leq 1-P \\
    \frac{2(1-p-g)}{1-p}    &\leq 1-\frac{2p-1}{1-p} \\
    g                       &\geq p/2 = \min\bigl\{p/3,1-p\bigr\} + \max\biggl\{\frac{p}{6},\frac{-2+3p}{2}\biggr\} \\
                            &\geq \min\bigl\{p/3,1-p\bigr\} + 1/8.
\end{align*}

\vspace{-1cm}

If, without loss of generality $v_1$ has another neighbor via a gray edge (call it $v_0$) but $v_2$ does not, then $v_0$ has weight $G+Px-y$ and since there is no $C_3$, then $v_0$, its gray-edge neighbors and $v_2$ partition a subset of the vertex set of $K$.
Thus, $G+Py=x$ and $G+Px\leq 2y$, so $(2+P)G\leq (2-P^2)x$.
Hence $P<1/\sqrt{2}$ and 

\vspace{-1.5cm}

\begin{align*}
    \x\bigl(v_0\bigr) + \x\bigl(v_1\bigr) + \x\bigl(v_2\bigr)   &\leq 1 \\
    \bigl[G+Px-y\bigr] + \bigl[G+P\bigl(G+Px-y\bigr)\bigr] + y   &\leq 1 \\
    (2+P)G + P(1+P)x - Py           &\leq 1 \\
    (2+P)G + P(1+P)x - (x-G)        &\leq 1 \\
    (3+P)G + (-1+P+P^2)x            &\leq 1 .
\end{align*}
If $-1+P+P^2\geq 0$, then 
\begin{align*}
    (3+P)G + (-1+P+P^2)x                        &\leq 1 \\
    (3+P)G + (-1+P+P^2)\frac{(2+P)G}{2-P^2}     &\leq 1 \\
    G                                           &\leq \frac{2-P^2}{4+3P} \\
    \frac{1-p-g}{1-p}                           &\leq \frac{2(1-p)^2-(2p-1)^2}{4(1-p)+3(2p-1)} \\
    g                                           &\geq \frac{p}{1+2p} \\
                                                &= \min\bigl\{p/3,1-p\bigr\} + \max\biggl\{\frac{p(1-p)}{3(1+2p)},\frac{-1+2p^2}{1+2p}\biggr\} \\
                                                &\geq \min\bigl\{p/3,1-p\bigr\} +5/112 .
\end{align*}
If $-1+P+P^2<0$, then $G+Px\leq 1-x$ and 
\begin{align*}
    (3+P)G + (-1+P+P^2)x                        &\leq 1 \\
    (3+P)G + (-1+P+P^2)\frac{1-G}{1+P}          &\leq 1 \\
    G                                           &\leq \frac{2-P^2}{4+3P} \\
    g                                           &\geq \frac{p}{1+2p} \\
                                                &\geq \min\bigl\{p/3,1-p\bigr\} +5/112 .
\end{align*}

Finally, suppose $v_1$ and $v_2$ each have two gray-edge neighbors.
Call them $v_0$ and $v_3$, respectively.
They have respective weights of $G+Px-y$ and $G+Py-x$.
The maximality of $x+y$ gives $G+Px\leq 2y$ and $G+Py\leq 2x$. 
Thus, $2G\leq (2-P)(x+y)$. 
Hence $P<2$ (equivalently $p<3/4$) and since the girth of the graph induced by gray edges is at least 6, the closed neighborhoods of $v_0$ and $v_3$ partition a subset of the vertex set of $K$.
\begin{align*}
    \bigl[G+(P+1)(G+Px-y)\bigr] + \bigl[G+(P+1)(G+Py-x)\bigr]   &\leq 1 \\
    2(2+P)G + \bigl(-1+P^2\bigr)(x+y)                                     &\leq 1 \\
\end{align*}
If $-1+P^2\geq 0$, then since $p<3/4$,
\begin{align*}
    2(2+P)G + \bigl(-1+P^2\bigr)(x+y)             &\leq 1 \\
    2(2+P)G + \bigl(-1+P^2\bigr)\frac{2G}{2-P}    &\leq 1 \\
    G                                   &\leq \frac{2-P}{6} \\
    \frac{1-p-g}{1-p}                   &\leq \frac{2(1-p)-(2p-1)}{6(1-p)} \\
    g                                   &\geq \frac{3-2p}{6} = \min\bigl\{p/3,1-p\bigr\} + \max\biggl\{\frac{3-4p}{6},\frac{-3+4p}{6}\biggr\} \\
                                        &\geq \min\bigl\{p/3,1-p\bigr\} .
\end{align*}
If $-1+P^2<0$, then since $\bigl(G+Px\bigr)+\bigl(G+Py\bigr)<1$ we have $x+y<(1-2G)/P$ and 
\begin{align*}
    2(2+P)G + \bigl(-1+P^2\bigr)(x+y)           &\leq 1 \\
    2(2+P)G + \bigl(-1+P^2\bigr)\frac{1-2G}{P}  &< 1 \\
    G                                           &< \frac{1+P-P^2}{2(1+2P)} \\
    \frac{1-p-g}{1-p}                           &< \frac{(1-p)^2+(2p-1)(1-p)-(2p-1)^2}{2(1-p)\bigl[(1-p)+2(2p-1)\bigr]} \\
    g                                           &> \frac{-1+3p-p^2}{2(-1+3p)} \\
                                                &= \min\bigl\{p/3,1-p\bigr\} + \max\biggl\{\frac{-3+11p-9p^2}{6(-1+3p)},\frac{1-5p+5p^2}{2(-1+3p)}\biggr\} \\
                                                &\geq \min\bigl\{p/3,1-p\bigr\} + 1/40 .
\end{align*}

This concludes the proof and establishes that $g_K(p)\geq\min\bigl\{p/3,1-p\bigr\}$ for all $p\in [0,1]$. 

\section{Proof of Theorem~\ref{thm:kword}}
\label{sec:proof:kword}

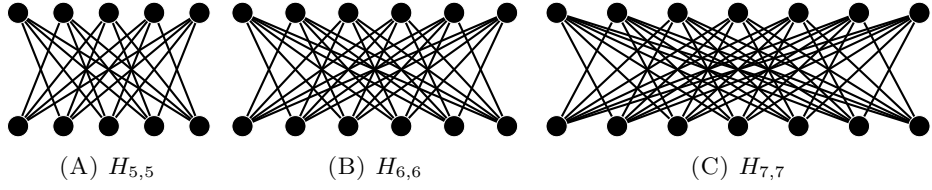
\begin{figure}
    \centering
    \ \hfill \
    \begin{subfigure}[b]{0.21\textwidth}
        \centering
        \begin{tikzpicture}[scale=1]
            \centering
            \def\n{5}
            \def\stretch{0.6}
            \useasboundingbox (-1.35,-0.15) rectangle (1.35,1.65);
            \foreach \i in {1,2}{
                \foreach \j in {1,...,\n}{
                    \coordinate (\i\j) at ($(\stretch*\n/2+\stretch/2-\stretch*\j,1.5*\i-1.5)$);
                }
            }
            \begin{pgfonlayer}{fore}
            \foreach \i in {1,2}{
                \foreach \j in {1,...,\n}{
                    \draw (\i\j) node[vtx] {};
                }
            }
            \end{pgfonlayer}
            \begin{pgfonlayer}{main}
                \foreach \j in {2,...,\n}{
                    \draw[edg] (11) -- (2\j);
                }
                \foreach \i in {2,...,\n-1}{
                    \pgfmathtruncatemacro{\prev}{\i-1};
                    \pgfmathtruncatemacro{\next}{\i+1};
                    \foreach \j in {1,...,\prev}{
                        \draw[edg] (1\i) -- (2\j);
                    }
                    \foreach \j in {\next,...,\n}{
                        \draw[edg] (1\i) -- (2\j);
                    }                }
                \foreach \j in {1,...,\n-1}{
                    \draw[edg] (1\n) -- (2\j);
                }            
            \end{pgfonlayer}
        \end{tikzpicture}
        \caption{$H_{5,5}$}
        \label{fig:H55}
    \end{subfigure}\ \hfill \
    \begin{subfigure}[b]{0.31\textwidth}
        \centering
        \begin{tikzpicture}[scale=1]
            \centering
            \def\n{6}
            \def\stretch{0.7}
            \useasboundingbox (-1.85,-0.15) rectangle (1.85,1.65);
            \foreach \i in {1,2}{
                \foreach \j in {1,...,\n}{
                    \coordinate (\i\j) at ($(\stretch*\n/2+\stretch/2-\stretch*\j,1.5*\i-1.5)$);
                }
            }
            \begin{pgfonlayer}{fore}
            \foreach \i in {1,2}{
                \foreach \j in {1,...,\n}{
                    \draw (\i\j) node[vtx] {};
                }
            }
            \end{pgfonlayer}
            \begin{pgfonlayer}{main}
                \foreach \j in {2,...,\n}{
                    \draw[edg] (11) -- (2\j);
                }
                \foreach \i in {2,...,\n-1}{
                    \pgfmathtruncatemacro{\prev}{\i-1};
                    \pgfmathtruncatemacro{\next}{\i+1};
                    \foreach \j in {1,...,\prev}{
                        \draw[edg] (1\i) -- (2\j);
                    }
                    \foreach \j in {\next,...,\n}{
                        \draw[edg] (1\i) -- (2\j);
                    }
                }
                \foreach \j in {1,...,\n-1}{
                    \draw[edg] (1\n) -- (2\j);
                }            
            \end{pgfonlayer}
        \end{tikzpicture}
        \caption{$H_{6,6}$}
        \label{fig:H66}
    \end{subfigure}\ \hfill \
    \begin{subfigure}[b]{0.41\textwidth}
        \centering
        \begin{tikzpicture}[scale=1]
            \centering
            \def\n{7}
            \def\stretch{0.8}
            \useasboundingbox (-2.50,-0.15) rectangle (2.50,1.65);
            \foreach \i in {1,2}{
                \foreach \j in {1,...,\n}{
                    \coordinate (\i\j) at ($(\stretch*\n/2+\stretch/2-\stretch*\j,1.5*\i-1.5)$);
                }
            }
            \begin{pgfonlayer}{fore}
            \foreach \i in {1,2}{
                \foreach \j in {1,...,\n}{
                    \draw (\i\j) node[vtx] {};
                }
            }
            \end{pgfonlayer}
            \begin{pgfonlayer}{main}
                \foreach \j in {2,...,\n}{
                    \draw[edg] (11) -- (2\j);
                }
                \foreach \i in {2,...,\n-1}{
                    \pgfmathtruncatemacro{\prev}{\i-1};
                    \pgfmathtruncatemacro{\next}{\i+1};
                    \foreach \j in {1,...,\prev}{
                        \draw[edg] (1\i) -- (2\j);
                    }
                    \foreach \j in {\next,...,\n}{
                        \draw[edg] (1\i) -- (2\j);
                    }
                }
                \foreach \j in {1,...,\n-1}{
                    \draw[edg] (1\n) -- (2\j);
                }            
            \end{pgfonlayer}
        \end{tikzpicture}
        \caption{$H_{7,7}$}
        \label{fig:H77}
    \end{subfigure}\ \hfill \
	\caption{Examples of crown graphs, $H_{n,n}$, for which $\RR\bigl(H_{n,n}\bigr)=\bigl\lceil n/2\bigr\rceil$.}
    \label{fig:forb:kword}
\end{figure}

From Theorem~\ref{thm:KP}, we know $\hword\subset\forb\{F_1,F_2,W_5\}$. 

Glen, the first author, and Pyatkin proved that there exist small graphs with large word-representation number.
For all $n\geq 2$, the \emph{crown graph $H_{n,n}$} is a bipartite graph with $n$ vertices in each part that is a complete bipartite graph with a perfect matching missing\footnote{A crown graph is not to be confused with a crown poset. A crown poset is a height-two poset which has a Hasse diagram that is an even cycle.}; see Figure~\ref{fig:forb:kword} for examples.
\begin{theorem}[Glen-K.-Pyatkin\cite{GKP}]
    If $n\geq 5$ then $\RR\bigl(H_{n,n}\bigr)=\lceil n/2\rceil$.
    That is, one needs $\lceil n/2\rceil$ copies of each letter to represent $H_{n,n}$ but not fewer.

    Hence, for all $k\geq 2$, the graph $H_{2k+2,2k+2}$ is not $k$-word representable.
\end{theorem}

So we note that for $k\geq 2$, $\hkword{k}\subset\forb\bigl\{F_1,F_2,H_{2k+2,2k+2}\bigr\}$.

Since nontrivial empty graphs have representation number 2 and cliques have representation number 1, if $k\geq 3$, then we can delete all the edges or add all the nonedges, resulting in an editing procedure with $\min\bigl\{|E(G)|,\tnct-|E(G)|\bigr\}$ operations. 

Hence, $\ed_{\hkword{k}}(p)\leq\min\bigl\{p,1-p\bigr\}$ and $\dist\bigl(n,\hword\bigr)\leq \frac{1}{4}\tnct$.
In addition if $\dist\bigl(G,\hkword{k}\bigr)\geq\bigl(\frac{1}{2}-\epsilon\bigr)\tnct$, then $\min\bigl\{|E(G)|,\tnct-|E(G)|\bigr\}\geq\bigl(\frac{1}{2}-\epsilon\bigr)\tnct$ and so $\bigl|E(G)\bigr|/\tnct\in \bigl[\tfrac{1}{2}-\epsilon,\tfrac{1}{2}+\epsilon\bigr]$.

As far as the lower bound, observe that 
\begin{itemize}
    \item $F_1\mapsto K(1,1)$,
    \item $F_2\mapsto K(0,2)$, and 
    \item $H_{2k+2,2k+2}\mapsto K(2,0)$.
\end{itemize}

As a result, any $K\in\KK_p(\HH)$ has no gray edges. 
By Corollary~\ref{cor:graydeg}, if $K$ has no gray edges, then $g_K(p)\geq\min\bigl\{p,1-p\}$ and the equality is not sharp for any $p\in (0,1)$ unless $p\leq 1/2$ and $K$ is a single white vertex or $p\geq 1/2$ and $K$ is a single black vertex. 

\section{Proof of Theorem~\ref{thm:comp}}
\label{sec:proof:comp}

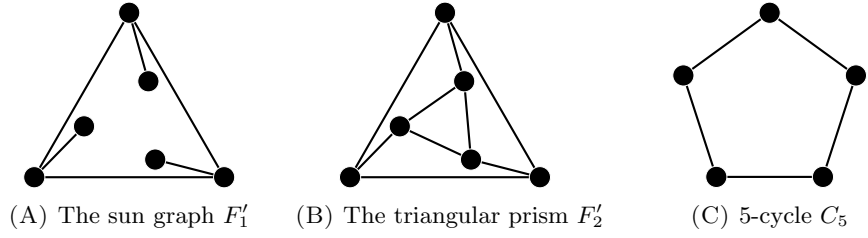
\begin{figure}
    \centering 
    \ \hfill \
    \begin{subfigure}[b]{0.30\textwidth}
        \centering
        \begin{tikzpicture}[scale=1]
            \centering
            \useasboundingbox (-1.6,-1.25) rectangle (1.6,1.25);
            \coordinate (offset) at (0,-0.4);
            \foreach \i in {1,2,3}{
                \coordinate (v\i) at ($(-30+120*\i:1.45)+(offset)$);
                \coordinate (w\i) at ($(-55+120*\i:0.60)+(offset)$);
            }
            \begin{pgfonlayer}{fore}
                \foreach \i in {1,2,3}{
                    \draw (v\i) node[vtx] {};
                    \draw (w\i) node[vtx] {};
                }
            \end{pgfonlayer}
            \begin{pgfonlayer}{main}
                \foreach \i in {1,2}{
                    \draw[edg] (v\i) -- (w\i);
                    \pgfmathtruncatemacro{\j}{\i+1};
                    \draw[edg] (v\i) -- (v\j);
                }
                \draw[edg] (v3) -- (w3);
                \draw[edg] (v1) -- (v3);
            \end{pgfonlayer}
        \end{tikzpicture}
        \caption{The sun graph $F'_1$}
        \label{fig:F1p}
    \end{subfigure}\ \hfill \
    \begin{subfigure}[b]{0.33\textwidth}
        \centering
        \begin{tikzpicture}[scale=1]
            \centering
            \useasboundingbox (-1.6,-1.25) rectangle (1.6,1.25);
            \coordinate (offset) at (0,-0.4);
            \foreach \i in {1,2,3}{
                \coordinate (v\i) at ($(-30+120*\i:1.45)+(offset)$);
                \coordinate (w\i) at ($(-55+120*\i:0.60)+(offset)$);
            }
            \begin{pgfonlayer}{fore}
                \foreach \i in {1,2,3}{
                    \draw (v\i) node[vtx] {};
                    \draw (w\i) node[vtx] {};
                }
            \end{pgfonlayer}
            \begin{pgfonlayer}{main}
                \foreach \i in {1,2}{
                    \draw[edg] (v\i) -- (w\i);
                    \pgfmathtruncatemacro{\j}{\i+1};
                    \draw[edg] (v\i) -- (v\j);
                    \draw[edg] (w\i) -- (w\j);
                }
                \draw[edg] (v3) -- (w3);
                \draw[edg] (v1) -- (v3);
                \draw[edg] (w1) -- (w3);
            \end{pgfonlayer}
        \end{tikzpicture}
        \label{fig:F2p}
        \caption{The triangular prism $F_2'$}
    \end{subfigure}\ \hfill \
    \begin{subfigure}[b]{0.30\textwidth}
        \centering
        \begin{tikzpicture}[scale=1]
            \centering
            \useasboundingbox (-1.6,-1.25) rectangle (1.6,1.25);
            \coordinate (offset) at (0,-0.15);
            \foreach \i in {1,2,3,4,5}{
                \coordinate (v\i) at ($(18+72*\i:1.20)+(offset)$);
            }
            \begin{pgfonlayer}{fore}
                \foreach \i in {1,2,3,4,5}{
                    \draw (v\i) node[vtx] {};
                }
            \end{pgfonlayer}
            \begin{pgfonlayer}{main}
                \foreach \i in {1,2,3,4}{
                    \pgfmathtruncatemacro{\j}{\i+1};
                    \draw[edg] (v\i) -- (v\j);
                }
                \draw[edg] (v1) -- (v5);
            \end{pgfonlayer}
        \end{tikzpicture}
        \caption{$5$-cycle $C_5$}
        \label{fig:C5}
    \end{subfigure}\ \hfill \
	\caption{Examples of non-comparability graphs.}
    \label{fig:forb:comp}
\end{figure}

There are two types of editing procedures to perform on a graph $G$:
\begin{itemize}
    \item For some $k\geq 2$, randomly partition $V(G)$ into $V_1,\ldots, V_k$, delete any edge with both endpoints in the same $V_i$ and add any nonedge with one endvertex in $V_i$ and the other in $V_j$ such that $|j-i|\geq 2$. 
    \item Add all nonedges.
\end{itemize}

In the first case, the resulting graph is a comparability graph because we can orient edges $\overrightarrow{v_iv_j}$ if $v_i\in V_i$, $v_j\in V_j$ and $i<j$.
In the second case, simply make it a transitive tournament, which corresponds to a chain poset. 

In addition, we can optimally weight the sizes $|V_1|,\ldots,|V_k|$ so as to minimize the expected number of edits.
As a result, the expected number of changes is $g_K(p)$ where $K$ is the CRG which consists of $k$ white vertices, the gray edges form a spanning path and all other edges are black.

As to the lower bound, Gallai~\cite{Gal} completely classified graphs that are not comparability graphs.
There are 8 infinite families along with 10 additional individual graphs.
We do not need all of these to compute the edit distance, we only use one infinite family and two additional graphs.
\begin{corollary}[Gallai~\cite{Gal}]
    Given the graphs $F_1'$ and $F_2'$ in Figure~\ref{fig:forb:comp}, we have
    $$
        \hcomp\subset\forb\bigl(\bigl\{F_1',F_2',C_5=\overline{C_5},\overline{C_6},\overline{C_7},\ldots\bigr\}\bigr) .
    $$
\end{corollary}

To see which CRGs we will use, we see which forbidden graphs map to which CRGs.  
\begin{proposition}
    The following embeddings exist:
    \begin{enumerate}[(a)]
        \item $F_1'\mapsto K(1,1)$. \label{it:comp:structure:k11}
        \item $F_2'\mapsto K(0,2)$. \label{it:comp:structure:k02}
        \item If $K$ has all of its vertices white, and all of its edges either black or gray such that the gray edges have a $K_{1,3}$ as a subgraph, then $F_1'\mapsto K$. \label{it:comp:structure:claw}
        \item If $t\in\{3,4,5\}$ and $K$ has all of its vertices white, and all of its edges either black or gray such that the gray edges have a $C_t$ as a subgraph, then $C_5\mapsto K$. \label{it:comp:structure:smcycle}
        \item If $t\geq 6$ and $K$ has all of its vertices white, and all of its edges either black or gray, such that the gray edges have a $C_t$ as a subgraph, then $\overline{C_t}\mapsto K$. \label{it:comp:structure:cycle}
    \end{enumerate}
    \label{prop:comp:structure}
\end{proposition}

\begin{proof}
    For \eqref{it:comp:structure:k11}, we observe that $F_1'$ is a split graph. 
    That is, $F_1'$ can be partitioned into two parts, one is an independent set and one is a clique.
    The independent set maps to the white vertex and the clique maps to the black vertex. 

    For \eqref{it:comp:structure:k02}, we observe that $\chi\bigl(\overline{F_2'}\bigr)=2$.
    Hence, $F_2'$ can be partitioned into two parts, each of which is a clique.
    Each of the two cliques maps to a different black vertex.

    For \eqref{it:comp:structure:claw}, we refer to the mapping in Figure~\ref{fig:F1toK13}.
    The graph $F_1'$ is on $6$ vertices. 
    The vertices of the independent set of size $3$, consisting of the degree-$1$ vertices, are all mapped to the center of the gray $K_{1,3}$. 
    The remaining degree-$3$ vertices are mapped to different leaves of the gray $K_{1,3}$. 

    For \eqref{it:comp:structure:smcycle}, we recall the mapping in Figure~\ref{fig:W5toC3C4} in the cases of $t\in\{3,4\}$.
    The vertices of the complement of $C_5$ are mapped as independent sets to the vertices of $C_t$.

    For \eqref{it:comp:structure:smcycle} in the case of $t=5$ and for \eqref{it:comp:structure:cycle} in the case of $t\geq 6$, it is easy to see that the vertices of $\overline{C_t}$ can be mapped to the vertices of a $t$-vertex $p$-core CRG where the gray edges induced a $t$-cycle such that the vertices are mapped to different vertices of the CRG and the nonedges are mapped to the gray edges. 
\end{proof}

\newcommand{\kpath}[1]{\widetilde{P}_{#1}}
The most important CRGs for $\hcomp$ are denoted $\kpath{k}$. 
\begin{definition}
    For $k\geq 2$, the CRG $\kpath{k}$ consists of $k$ white vertices and all edges black except for a set of $k-1$ gray edges that form a spanning path.
    \label{defn:kpath}
\end{definition}

\begin{lemma}
    For $p\in (0,1)$, the only $p$-core CRGs in $\KK\bigl(\hcomp\bigr)$ for which $g_K(p)=\ed_{\hcomp}(p)$ are $\kpath{2}=K(2,0)$ if $p\leq 1/2$ and $\kpath{2},\kpath{3},\ldots$ and $K(0,1)$ if $p>1/2$. 
    \label{lem:classify}
\end{lemma}

\begin{proof}
Suppose that $K\in\KK_p(\HH)$ for some value of $p$. 
Proposition~\ref{prop:comp:structure}\eqref{it:comp:structure:k11} gives that $F_1'\mapsto K(1,1)$, so by Theorem~\ref{thm:MT}\eqref{it:MT:vwvb} either $V(K)=\VW(K)$ or $V(K)=\VB(K)$. 

First, let $p\leq 1/2$. 
In this case, if $V(K)=\VB(K)$, then since Proposition~\ref{prop:comp:structure}\eqref{it:comp:structure:k11} gives that $F_2'\mapsto K(0,2)$, there can be no gray edge and Corollary~\ref{cor:graydeg} gives that $g_K(p)>p\geq p/2$.
If $V(K)=\VW(K)$, then since $C_5\mapsto K(3,0)$, Theorem~\ref{thm:MT}\eqref{it:MT:psmall:vw} gives that $|\VW(K)|\leq 2$ and Corollary~\ref{cor:components} gives that $g_K(p)\geq p/2$, with equality if and only if $K=K(2,0)$.

Next, let $p>1/2$.
Again, if $V(K)=\VB(K)$ then since Proposition~\ref{prop:comp:structure}\eqref{it:comp:structure:k11} gives that $F_2'\mapsto K(0,2)$, there can be no gray edge.
Thus, Theorem~\ref{thm:MT}\eqref{it:MT:plarge:vb} gives that $|\VB(K)|=1$, hence $K=K(0,1)$ and $g_K(p)=1-p$.

Finally, we have $V(K)=\VW(K)$.
By Theorem~\ref{thm:MT}\eqref{it:MT:plarge:vw}, all edges of $K$ are either gray or black.
From Proposition~\ref{prop:comp:structure}\eqref{it:comp:structure:claw},~\eqref{it:comp:structure:smcycle}, and~\eqref{it:comp:structure:cycle}, the graph induced by the gray edges of $K$ has maximum degree 2 and has no cycle.

Thus the graph in $K$ induced by the gray edges must consist of disjoint paths.
However, if we change black edges to gray edges in a CRG $K$ to form the CRG $K'$, it must be the case that $g_K(p)\geq g_{K'}(p)$.
Thus, by adding gray edges to connect smaller paths together, we may assume that every CRG under consideration must have its gray edges induce a spanning path. 

Thus, we have established that 
$$ \ed_{\hcomp}(p)=\min\biggl\{\frac{p}{2},\min\Bigl\{g_{\kpath{k}}(p) : k\geq 3\Bigr\},1-p\biggr\} $$
and that $\ed_{\hcomp}(p)=p/2$ for all $p\in [0,1/2]$.
\end{proof}

It remains to analyze the various functions $g_{\kpath{k}}(p)$ for $k\geq 2$ and $p>1/2$. 
Consider the CRG $\kpath{k}$ with vertices $v_1v_2\cdots v_k$ inducing a path such that the edges $v_iv_{i+1}$ are gray, for all $i\in\{1,\ldots,k-1\}$, and all other edges are black. 

As in the proof of Proposition~\ref{prop:star}, we will do a change of variables in order to simplify things.
Let 
$$ g=g_K(p), \qquad G=\frac{1-p-g}{1-p} , \qquad P=\frac{2p-1}{1-p} . $$

For $p\in (1/2,1)$ and any $p$-core CRG $K$ on $k$ vertices, let $\widehat{{\bf M}}_K(P) = {\bf J} - \frac{1}{1-p}{\bf M}_K(p) = {\bf J} - (P+2){\bf M}_K\bigl(\frac{P+1}{P+2}\bigr)$, where ${\bf J}$ is the $k\times k$ all-ones matrix. 

In the case where $K=\kpath{k}$,
$$  \bigl(\widehat{m}_{\kpath{k}}\bigr)_{ij}  =   \begin{cases}
                    -P,  &\mbox{if $i=j$;} \\
                    1,  &\mbox{if $|i-j|=1$; and} \\
                    0,    &\mbox{else.}
                \end{cases}
$$

The advantage to considering $\widehat{{\bf M}}_K(P)$ for $K=\kpath{k}$ is that it is a symmetric tridiagonal Toeplitz matrix.\footnote{A Toeplitz matrix is one that is constant on each of its diagonals.} 
Furthermore if $\x$ is the optimal vector given by Theorem~\ref{thm:graydeg} for which $\x>{\bf 0}$, and $\x^T{\bf 1}=1$ and ${\bf M}_K(p)\cdot\x = g_K(p){\bf 1}$, then this $\x$ also obeys $\widehat{{\bf M}}_K(P)\cdot\x=G{\bf 1}$.

Lemma~\ref{lem:Toeplitz} gives the eigenvalues and eigenvectors of $\widehat{{\bf M}}_{\kpath{k}}(P)$.

\begin{lemma}
    Let $k\geq 2$ and $\widehat{{\bf M}}$ be a $k\times k$ symmetric tridiagonal Toeplitz matrix such that the diagonal entries are $-P$ and the super-diagonal and sub-diagonal entries are all equal to $1$. 
    Then, for each $a\in \{1,\ldots,k\}$, $\lambda_a=2\cos\bigl(\frac{a\pi}{k+1}\bigr)-P$ is an eigenvalue of $\widehat{{\bf M}}$ such that a corresponding unit eigenvector is 
    $$ 
        \bigl({\bf w}_a\bigr)_b = \biggl(\sqrt{\frac{2}{k+1}}\sin\Bigl(\frac{a\pi}{k+1}\cdot b\Bigr)\biggr)_b . 
    $$
    Furthermore,
    $$
        {\bf w}_a^T{\bf 1} 
        =   \begin{cases}
                0,  &   \mbox{if $a$ is even;} \\
                \sqrt{\frac{2}{k+1}}\, \frac{1+\cos\bigl(\frac{a\pi}{k+1}\bigr)}{\sin\bigl(\frac{a\pi}{k+1}\bigr)}, & \mbox{if $a$ is odd.}
            \end{cases}
    $$
    \label{lem:Toeplitz}
\end{lemma}

Lemma~\ref{lem:Toeplitz} is proven in Appendix~\ref{sec:Toeplitz}.

\begin{lemma}
    Let $K$ be a CRG with $V(K)=\VW(K)$ and $E(K)=\EB(K)\sqcup\EG(K)$ and let $p>1/2$. 
    Let $\widehat{{\bf M}}={\bf J}-\frac{1}{1-p}{\bf M}_K(p)$. 
    Let ${\bf w}$ be an eigenvector of $\widehat{{\bf M}}$ corresponding to the eigenvalue $\lambda$. 
    If $K$ is $p$-core, then 
    $$
        \lambda     
        <       \biggl(1-\frac{g_K(p)}{1-p}\biggr) \frac{\bigl({\bf w}^T{\bf 1}\bigr)^2}{\|{\bf w}\|_2^2} .
    $$
    \label{lem:eigencore}
\end{lemma}

\begin{proof}
    Let $\x$ be the optimal weight from Theorem~\ref{thm:graydeg} and let $g=g_K(p)$. 
    Observe that $\widehat{{\bf M}}\cdot\x=\bigl(1-\frac{g}{1-p}\bigr){\bf 1}$ and $\x^T\cdot\widehat{{\bf M}}\cdot\x=1-\frac{g}{1-p}$
    Let $\epsilon>0$ be sufficiently small such that $1+\epsilon{\bf w}^T\cdot{\bf 1}>0$ and $\x+\epsilon{\bf w}\geq {\bf 0}$.

    Consider the vector ${\bf v}=\frac{1}{1+\epsilon\cdot{\bf w}^T\cdot{\bf 1}}\bigl(\x+\epsilon{\bf w}\bigr)$.
    By the optimality of $\x$, it must be the case that 
    \begin{align*}
        \x^T\widehat{{\bf M}}\x 
        &> {\bf v}^T\widehat{{\bf M}}{\bf v} = \frac{1}{\bigl(1+\epsilon{\bf w}^T{\bf 1}\bigr)^2} \bigl(\x+\epsilon{\bf w}\bigr)^T \widehat{{\bf M}} \bigl(\x+\epsilon{\bf w}\bigr) \\
        \bigl(1+\epsilon\cdot{\bf w}^T{\bf 1}\bigr)^2 \x^T\widehat{{\bf M}}\x 
        &> \x^T\widehat{{\bf M}}\x + 2\epsilon \Bigl(1-\frac{g}{1-p}\Bigr){\bf w}^T{\bf 1} + \epsilon^2\lambda\bigl\|{\bf w}\bigr\|_2^2 \\
        \Bigl(2\epsilon{\bf w}^T{\bf 1}+\epsilon^2\bigl({\bf w}^T{\bf 1}\bigr)^2\Bigr) \Bigl(1-\frac{g}{1-p}\Bigr) 
        &> 2\epsilon \Bigl(1-\frac{g}{1-p}\Bigr)\bigl({\bf w}^T{\bf 1}\bigr) + \epsilon^2\lambda\bigl\|{\bf w}\bigr\|_2^2 \\
        \bigl({\bf w}^T{\bf 1}\bigr)^2 \Bigl(1-\frac{g}{1-p}\Bigr) 
        &> \lambda\bigl\|{\bf w}\bigr\|_2^2
    \end{align*}
\end{proof}

Now we use Lemma~\ref{lem:kpath} to establish exactly what the structure of $\ed_{\hcomp}(p)$ is.

\begin{lemma}
    For $k\geq 2$, let the real number $p_k$ be as in Definition~\ref{defn:pk} and the CRG $\kpath{k}$ be as in Definition~\ref{defn:kpath}. 
    The CRG $\kpath{k}$ is not $p$-core for $p\in (0,p_k]$ and is $p$-core for all $p\in \bigl(p_k,p_{k+1}\bigr]$.
    Moreover, $g_{\kpath{k}}(p)>1-p$ for all $p\in \bigl[3/4,1\bigr]$.

    Consequently, $\ed_{\hcomp}(p)=g_{\kpath{k}}(p)$ for $k\geq 2$ and for all $p\in\bigl(p_k,p_{k+1}\bigr]$ and $\ed_{\hcomp}(p)=1-p$ for all $p\in\bigl[3/4,1\bigr]$.
    \label{lem:kpath}
\end{lemma}

The reader may note that we do not resolve whether $\kpath{k}$ is $p$-core for $p\in\bigl[3/4,1\bigr]$.
This is because we establish that $g_{\kpath{k}}(p)>1-p$ for $p$ in that range and so the CRGs $\kpath{k}$ are moot for the problem we are trying to solve. 

\begin{proof}
Let $k\geq 2$.
First, we prove that $K=\kpath{k}$ is not $p$-core for $p\in\bigl(0,p_k\bigr]$.
Let $P=\frac{2p-1}{1-p}$ and $G=G_K(P)=1-\frac{g_K(p)}{1-p}$. 
Lemma~\ref{lem:eigencore} gives that if $\kpath{k}$ is $p$-core then, 
\begin{align*} 
    \lambda_2   
    &<  \biggl(1-\frac{g_K(p)}{1-p}\biggr)\frac{\bigl({\bf w}^T{\bf 1}\bigr)^2}{\|{\bf w}\|_2^2} \\
    2\cos\Bigl(\frac{2\pi}{k+1}\Bigr)-P 
    &<  0 ,
\end{align*}
because Lemma~\ref{lem:Toeplitz} gives ${\bf w}^T{\bf 1}=0$.
Hence if $\kpath{k}$ is $p$-core, then $P>2\cos\bigl(\frac{2\pi}{k+1}\bigr)$ (equivalently $p>p_k$). 

Second, we prove that $\kpath{k}$ is $p$-core for $p\in\bigl(p_k,p_{k+1}\bigr]$ by induction on $k\geq 2$. 
For $k=2$, $\kpath{2}=K(2,0)$ is $p$-core for $p\in \bigl(0,1\bigr)=\bigl(p_2,1\bigr)\supset \bigl(p_2,p_3\bigr]=\bigl(p_2,1/2\bigr]$. 
Now suppose $k\geq 3$.

The deletion of any set of vertices of $\kpath{k}$ yields a sub-CRG for which the gray edges induce a subgraph consisting of disjoint paths. 
If such a sub-CRG $K'$ satisfies $g_{K'}(p)=g_{K}(p)$, then we can add gray edges to form a CRG $K''$ whose gray edges form a spanning path and for which $g_{K''}(p)\leq g_{K'}(p)=g_{K}(p)$. 
But since $K''$ is a sub-CRG of $K$, we may assume that $K'$ itself is $\kpath{\ell}$ for some $\ell\leq k$.
Moreover, by the inductive hypothesis, $\kpath{k-1}$ is $p$-core for $p\in\bigl(p_{k-1},p_k\bigr]$. 
Therefore, in order to verify $p$-coreness of $\kpath{k}$, it suffices to show that $g_{\kpath{k}}(p)<g_{\kpath{k-1}}(p)$ for all $p\in \bigl(p_k,p_{k+1}\bigr]$.
Observe that $p\in \bigl(p_k,p_{k+1}\bigr]$ implies that $P\in\Bigl(2\cos\bigl(\frac{2\pi}{k+1}\bigr), 2\cos\bigl(\frac{2\pi}{k+2}\bigr)\Bigr)$.

To that end, observe that the fact that $\widehat{{\bf M}}\x=G{\bf 1}$ implies the following recurrence relation:
\begin{align*}
    x_{a-1} - P x_a + x_{a+1}   &=  G,  \qquad \forall a\in\{1,\ldots,k\} \\
    x_0 = x_{k+1}   &= 0
\end{align*}
The solution to such a recurrence is of the form $x_a=C\bigl(A r_1^a + B r_2^a + 1)$, where \begin{align*}
    r_1     &=  \frac{P}{2}+\frac{i}{2}\sqrt{4-P^2} \qquad 
    & A     &=  -\frac{1-r_2^{k+1}}{r_1^{k+1}-r_2^{k+1}} \\
    r_2     &=  \frac{P}{2}-\frac{i}{2}\sqrt{4-P^2}     \qquad 
    & B     &=  \frac{1-r_1^{k+1}}{r_1^{k+1}-r_2^{k+1}} \\
    && C    &=  \frac{G}{2-P}
\end{align*} 

The roots $r_1$ and $r_2$ form a conjugate pair, so $r_1r_2=1$. 
Since $p\in\bigl(p_k,p_{k+1}\bigr]\subset\bigl(1/2,3/4\bigr)$, it follows that we may choose a $\theta\in\bigl(\frac{2\pi}{k+2},\frac{2\pi}{k+1}\bigr]\subset\bigl(0,\pi/2\bigr)$ such that $P=2\cos \theta$ and for any positive integer $a$, $r_1^a+r_2^a=2\cos (a\theta)$, and $r_1^a-r_2^a=2i\sin(a\theta)$, where $i=\sqrt{-1}$. 

First, we will compute $\x=\bigl(x_a\bigr)_{a=1}^k$ in terms of $P$ and $G$ and then use $\sum_{a=1}^kx_a=1$ to compute $G$.

Using the recurrence,
\begin{align*}
    x_a     
    &=  \frac{G}{2-P} \biggl[-\frac{1-r_2^{k+1}}{r_1^{k+1}-r_2^{k+1}} r_1^a + \frac{1-r_2^{k+1}}{r_1^{k+1}-r_1^{k+1}} r_2^a + 1\biggr] \\
    &=  \frac{G}{2-P} \biggl[1-\frac{r_1^a-r_2^a+r_1^{k+1-a}-r_2^{k+1-a}}{r_1^{k+1}-r_2^{k+1}}\biggr] \\
    &=  \frac{G}{2-P} \biggl[1-\frac{\sin\bigl(a\theta\bigr)+\sin\bigl((k+1-a)\theta\bigr)}{\sin\bigl((k+1)\theta\bigr)}\biggr].  
\end{align*}

We can use the trigonometric identities \eqref{eq:sin:sums}, \eqref{eq:sin:doubleangle},  \eqref{eq:cos:anglediff}, and \eqref{eq:tan:halfangle} which are listed in Appendix~\ref{sec:trig}, in order to simplify the expression. 
\begin{align}
    x_a     
    &=  \frac{G}{2-P} \biggl[1-\frac{\cos\bigl(\frac{k+1}{2}\theta\bigr)\cos\bigl(a\theta\bigr) + \sin\bigl(\frac{k+1}{2}\theta\bigr)\sin\bigl(a\theta\bigr)}{\cos\bigl(\frac{k+1}{2}\theta\bigr)}\biggr] \nonumber \\
    &=  \frac{G}{2-P} \biggl[1 - \cos\bigl(a\theta\bigr) - \sin\bigl(a\theta\bigr) \tan\bigl(\tfrac{k+1}{2}\theta\bigr)\biggr] \label{eq:hcomp:xa} \\
    &=  \frac{G}{2-P} \biggl[1 - \cos\bigl(a\theta\bigr) - \sin\bigl(a\theta\bigr) \frac{\sin\bigl((k+1)\theta\bigr)}{1+\cos\bigl((k+1)\theta)}\biggr]. \nonumber
\end{align}

Given $G$, we can compute the values of $x_a$ using the Chebyshev polynomials of the first and second kind, $\{T_n\}_{n\geq 0}$ and $\{U_n\}_{n\geq 0}$, respectively. 
This is done via the identities $\cos\bigl(n\theta\bigr)=T_n\bigl(\cos\theta\bigr)=T_n\bigl(P/2\bigr)$ and $\sin\bigl(n\theta\bigr)=\sin\theta\cdot U_{n-1}\bigl(\cos\theta\bigr)=\frac{1}{2}\sqrt{4-P^2}\cdot U_{n-1}\bigl(P/2\bigr)$.
See Appendix~\ref{sec:Cheb}.

Now we turn to computing $G$ itself.
Since $\sum_{a=1}^k x_a=1$,

\vspace{-1cm}

\begin{align*}
    1   
    &=  \sum_{a=1}^k C\bigl(Ar_1^a + Br_2^a +1\bigr) \\
    &=  \frac{G}{2-P}\biggl[-\frac{1-r_2^{k+1}}{r_1^{k+1}-r_2^{k+1}}\cdot\frac{r_1-r_1^{k+1}}{1-r_1} + \frac{1-r_1^{k+1}}{r_1^{k+1}-r_2^{k+1}}\cdot\frac{r_2-r_2^{k+1}}{1-r_2} + k\biggr] \\
    &=  \frac{G}{2-P}\biggl[\frac{-\bigl(1-r_2^{k+1}\bigr)\bigl(r_1-r_1^{k+1}\bigr)\bigl(1-r_2\bigr) + \bigl(1-r_1^{k+1}\bigr)\bigl(r_2-r_2^{k+1}\bigr)\bigl(1-r_1\bigr)}{
    \bigl(r_1^{k+1}-r_2^{k+1}\bigr)\bigl(1-r_1\bigr)\bigl(1-r_2\bigr)} + k\biggr] \\
    &=  \frac{G}{2-P}\biggl[\frac{4i\sin\bigl((k+1)\theta)-4i\sin\bigl(k\theta\bigr)-4i\sin\theta}{2i\sin\bigl((k+1)\theta\bigr)\bigl(2-2\cos\theta\bigr)} + k\biggr] \\
    &=  \frac{G}{2-P}\biggl[\frac{\sin\bigl((k+1)\theta)-\sin\bigl(k\theta\bigr)-\sin\theta}{\sin\bigl((k+1)\theta\bigr)\bigl(1-\cos\theta\bigr)} + k\biggr] \\
    G   &=  2\bigl(1-\cos\theta\bigr)^2\biggl[1-\frac{\sin\bigl(k\theta\bigr)+\sin\theta}{\sin\bigl((k+1)\theta\bigr)} + k\bigl(1-\cos\theta\bigr)\biggr]^{-1} .
\end{align*}
Again, we use the trigonometric identities \eqref{eq:sin:sums}, \eqref{eq:sin:doubleangle},  \eqref{eq:cos:anglediff}, and \eqref{eq:tan:halfangle} in Appendix~\ref{sec:trig}.

\vspace{-0.5cm}

\begin{align}
    G   
    &=  2\bigl(1-\cos\theta\bigr)^2\biggl[1-\frac{2\sin\bigl(\frac{k+1}{2}\theta\bigr)\cos\bigl(\frac{k-1}{2}\theta\bigr)}{2\sin\bigl(\frac{k+1}{2}\theta\bigr)\cos\bigl(\frac{k+1}{2}\theta\bigr)} + k\bigl(1-\cos\theta\bigr)\biggr]^{-1} \nonumber \\
    &=  2\bigl(1-\cos\theta\bigr)^2\biggl[1-\frac{\cos\bigl(\frac{k-1}{2}\theta\bigr)}{\cos\bigl(\frac{k+1}{2}\theta\bigr)} + k\bigl(1-\cos\theta\bigr)\biggr]^{-1} \nonumber \\
    &=  2\bigl(1-\cos\theta\bigr)^2\biggl[1-\frac{\cos\bigl(\frac{k-1}{2}\theta\bigr)}{\cos\bigl(\frac{k+1}{2}\theta\bigr)} + k\bigl(1-\cos\theta\bigr)\biggr]^{-1} \nonumber \\
    &=  2\bigl(1-\cos\theta\bigr)\biggl[\bigl(k+1\bigr)-\frac{\sin\theta}{1-\cos\theta}\cdot\tan\bigl(\tfrac{k+1}{2}\theta\bigr)\biggr]^{-1} \label{eq:hcomp:G} \\
    &=  2\bigl(1-\cos\theta\bigr)\biggl[\bigl(k+1\bigr)-\frac{\sin\theta}{1-\cos\theta}\cdot\frac{1-\cos\bigl((k+1)\theta\bigr)}{\sin\bigl((k+1)\theta\bigr)}\biggr]^{-1} \nonumber \\
    &=  \bigl(2-P\bigr)\biggl[\bigl(k+1\bigr)-\frac{2}{2-P}\cdot\frac{1-T_{k+1}\bigl(P/2\bigr)}{U_k\bigl(P/2\bigr)}\biggr]^{-1}. \nonumber 
\end{align}

See Appendix~\ref{sec:gfunc} for the values of $G_K(P)$ and $g_K(p)$ where $K=\kpath{k}$ and $p\in\bigl(p_k,p_{k+1}\bigr]$, as computed by Mathematica. 

Now we need to show that $\kpath{k}$ is, indeed $p$-core for $p\in\bigl(p_k,p_{k+1}\bigr)$ by computing $\frac{2-2\cos\theta}{G_{\kpath{k-1}}}-\frac{2-2\cos\theta}{G_{\kpath{k}}}$ using \eqref{eq:hcomp:G} and the trigonometric identities \eqref{eq:tan:diffs}, and \eqref{eq:tan:halfangle} listed in Appendix~\ref{sec:trig}.
Note that, since $k\geq 3$, then $\theta\in\bigl(\frac{2\pi}{k+2},\frac{2\pi}{k+1}\bigr]\subset\bigl(0,\pi/2\bigr)$. 
Turning to \eqref{eq:hcomp:G},
\begin{align}
    \frac{2-2\cos\theta}{G_{\kpath{k-1}}}-\frac{2-2\cos\theta}{G_{\kpath{k}}(P)} 
    &= \biggl[k-\frac{\sin\theta}{1-\cos\theta}\cdot\tan\Bigl(\frac{k}{2}\theta\Bigr)\biggr] - \biggl[k+1-\frac{\sin\theta}{1-\cos\theta}\cdot\tan\Bigl(\frac{k+1}{2}\theta\Bigr)\biggr] \nonumber \\
    &= -1+\frac{\sin\theta}{1-\cos\theta}\biggl[\tan\Bigl(\frac{k+1}{2}\theta\Bigr)-\tan\Bigl(\frac{k}{2}\theta\Bigr)\biggr] \nonumber \\
    &= -1+\frac{\sin\theta}{1-\cos\theta}\cdot\tan\Bigl(\frac{\theta}{2}\Bigr)\biggl[1+\tan\Bigl(\frac{k+1}{2}\theta\Bigr)\cdot\tan\Bigl(\frac{k}{2}\theta\Bigr)\biggr] \nonumber \\
    &= \tan\Bigl(\frac{k+1}{2}\theta\Bigr)\cdot\tan\Bigl(\frac{k}{2}\theta\Bigr) . \label{eq:GkGkm1}
\end{align}

Note that, $\sin\theta,\cos\theta,\tan\bigl(\tfrac{\theta}{2}\bigr)\in (0,1)$, thus
\begin{align*}
    \tfrac{k+1}{2}\theta \in \bigl(\pi-\tfrac{\pi}{k+2},\pi\bigr] ,
    \qquad \mbox{ and } \qquad
    \tfrac{k}{2}\theta \in \bigl(\pi-\tfrac{2\pi}{k+2},\pi-\tfrac{\pi}{k+1}\bigr] .
\end{align*}
Hence, for $k\geq 2$, $\tan\bigl(\tfrac{k}{2}\theta\bigr)<\tan\bigl(\tfrac{k+1}{2}\theta\bigr)\leq 0$ with equality if and only if $\theta=\frac{2\pi}{k+1}$ and so if $\theta\in\bigl(\frac{2\pi}{k+2}, \frac{2\pi}{k+1}\bigr)$, 
\begin{align*}
    \frac{2-2\cos\theta}{G_{\kpath{k-1}}}      &> \frac{2-2\cos\theta}{G_{\kpath{k}}} \\
    G_{\kpath{k-1}}                  &< G_{\kpath{k}} \\
    1-\frac{g_{\kpath{k-1}}(p)}{1-p}    &< 1-\frac{g_{\kpath{k}}(p)}{1-p} \\
    g_{\kpath{k-1}}(p)                  &> g_{\kpath{k}}(p) ,
\end{align*}
which establishes that $\kpath{k}$ is $p$-core for $p\in\bigl(p_k,p_{k+1}\bigr)$.

In fact, if one looks carefully, we establish that $\kpath{k}$ is $p$-core as long as $\frac{k}{2}\theta>\frac{\pi}{2}$, which is equivalent to $P\in \Bigl(2\cos\bigl(\frac{2\pi}{k+1}\bigr),2\cos\bigl(\frac{\pi}{k}\bigr)\Bigr)$, which is equivalent to $p\in\bigl(p_k,p_{2k-1}\bigr)$.

Third, let $p\in [3/4,1)$ and consider $K=\kpath{k}$.
As we discussed before, if $K$ is not $p$-core, then there is an $\ell<k$ such that $g_{\kpath{\ell}}(p)=g_{\kpath{k}}(p)$. 
But by changing one edge from black to gray, we get a CRG that consists of white vertices, and all black edges except for a spanning cycle induced by gray edges, denote it $\widetilde{C}_{\ell}$. 
Thus $g_{\widetilde{C}_{\ell}}(p)<g_{\kpath{k}}(p)$. 

Since $\x=\frac{1}{\ell}{\bf 1}$ gives that $M_{\widetilde{C}_{\ell}}(p)\cdot\x$ is a constant vector and $\widetilde{C}_{\ell}$ is $p$-core, then by Theorem~\ref{thm:graydeg}, it is the unique vector satisfying
$$ g_{\widetilde{C}_{\ell}}(p) = \frac{p}{\ell}+(\ell-3)\frac{1-p}{\ell} = (1-p) + \frac{4p-3}{\ell} \geq 1-p, $$
which establishes that $g_{\kpath{k}}(p)>1-p$ for $p\in [3/4,1)$.  

Fourth, we observe that since $\ed_{\hcomp}(p)$ is continuous and concave down, and that $g_{\kpath{4}}(p)=\frac{-1+3p-p^2}{2(-1+3p)}$ is increasing for $p\in \bigl(p_4,p_5\bigr]=\bigl(\frac{\sqrt{5}-1}{2},\frac{2}{3}\bigr]$ and $g_{\kpath{5}}(p)=\frac{-2+4p-p^3}{-4+6p+3p^2}$ is decreasing for $p\in \bigl(p_5,p_6\bigr]\supset\bigl(\frac{2}{3},0.692\bigr]$.
Hence, $\ed_{\hcomp}(p)$ achieves its maximum at $p=2/3$ and that maximum is $\frac{5}{18}$.

Fifth, note that $g_{\kpath{4}}(p)\leq\frac{5}{18}-\frac{1}{2}\bigl(\frac{2}{3}-p\bigr)^2$ for all $p\in\bigl(p_4,2/3\bigr]$ and $g_{\kpath{5}}(p)\leq \frac{5}{18}-\frac{1}{12}\bigl(p-\frac{2}{3}\bigr)$ for all $p\in \bigl(2/3,p_6\bigr]$.
Thus, if $g>5/18-\epsilon$ and $\epsilon<0.001<\min\Bigl\{\frac{1}{2}\bigl(\frac{2}{3}-p_4\bigr)^2, \frac{1}{12}\bigl(p_6-\frac{2}{3}\bigr)\Bigr\}$, then
$$
    \frac{2}{3}-\sqrt{2\epsilon} < p < \frac{2}{3}+12\epsilon .
$$

This completes the proof of Lemma~\ref{lem:kpath}.
\end{proof}

Furthermore, this also completes the proof of Theorem~\ref{thm:comp}, which establishes that $\ed_{\hcomp}(p)$ is equal to $p/2$ for $p\in [0,1/2]$, is equal to $g_{\kpath{k}}(p)$ for $p\in [p_k,p_{k+1}]$ for all $k\geq 3$, and is equal to $1-p$ for $p\in [3/4,1]$. 

\section{Remarks}

\subsection{Transitions between CRGs in $\hword$}

In the case of $\ed_{\hword}(p)=\min\bigl\{p/3,1-p\bigr\}$ if $p\in (0,3/4)\cup (3/4,1)$, then there is a single $p$-core CRG $K$ for which $\ed_{\hword}(p)=g_K(p)$.
At $p=3/4$, however, there is an infinite number of CRGs for which $\ed_{\hword}(3/4)=1/4=g_K(3/4)$. 
In particular, consider a CRG $K$ on $k$ vertices such that all vertices are white, all edges are black and gray and the gray edges of $K$ induce a spanning cycle, then $g_K(p)=1-p+\frac{4p-3}{k}$ for $p\in [3/4,1]$ and $g_K(3/4)=1/4$. 

\subsection{Transitions between CRGs in $\hcomp$}
In the case of $g_{\hcomp}(p)$, we need some definitions from Cox, McGinnis and the second author~\cite{CMM}. 
Given a nontrivial herediary property $\HH$, we say that $p\in [0,1]$ is a \emph{regular point} if there is an $\epsilon>0$ and a finite set of CRGs $\KK'\subseteq\KK(\HH)$ such that 
$$ \ed_{\HH}(p) = \min_{K\in\KK} g_K(p), \qquad \mbox{for every $p\in\bigl(p_0-\epsilon,p_0+\epsilon\bigr)\cap \bigl[0,1\bigr]$.} $$
Otherwise, $p$ is called an \emph{accumulation point}. 

There are many examples of hereditary properties for which $0$ is an accumulation point such as $\forb\bigl(K_{3,3}\bigr)$ found by Marchant and Thomason~\cite{MT}. 
Observe that if $p$ is an accumulation point for $\HH=\bigcap_{H\in\FF(\HH)}\forb\bigl(H\bigr)$, then $1-p$ is an accumulation point for $\bigcap_{H\in\FF(\HH)}\forb\bigl(\overline{H}\bigr)$.

The question of whether a number other than $0$ or $1$ could be an accumulation point was established in the affirmative in~\cite[Theorem 1.5]{CMM} where it was shown that $1/4$ is an accumulation point of the hereditary property $\forb\bigl(\bigl\{K_{1,4},C_5,C_6,C_7,\ldots\bigr\}\bigr)$. 
This is very similar to the the complement of $\hcomp$, that is, it is similar to $\forb\bigl(\bigl\{\overline{F_1'},\overline{F_2'},C_5,C_6,C_7,\ldots\bigr\}\bigr)$ which has an accumulation point of $1/4$ also. 

The second author and Riasanovsky~\cite[Theorem 39]{MR} established that there are no accumulation points in $\bigl(2-\varphi,\varphi-1\bigr)=\bigl(0.382\ldots,0.618\ldots\bigr)$, where $\varphi$ is the golden ratio.
This leads to several questions and conjectures.

\begin{question}[\cite{CMM}, Question 5.2]
    Suppose that $\HH = \forb(\FF)$ where $\FF$ is finite. Can $\HH$ have any accumulation points in $\bigl(0, 1\bigr)$?
\end{question}

\begin{conjecture}[\cite{CMM}, Conjecture 5.3]
    For any $s, t \in \mathbb{N}$, $\forb\bigl(K_{s,t}\bigr)$ has no accumulation points in the interval $\bigl(0, 1\bigr]$.
\end{conjecture}

\begin{conjecture}[\cite{CMM}, Conjecture 5.4]
    Any non-trivial hereditary property has only finitely many accumulation points.
\end{conjecture}

\begin{conjecture}[\cite{CMM}, Conjecture 5.5]
    For any non-trivial hereditary property, either $\bigl(0, 1/2\bigr]$ or $\bigl[1/2, 1\bigr)$ is free of accumulation points.
\end{conjecture}

\begin{question}
    Does there exist a $p\in\bigl(1/4,2-\varphi\bigr]\cup\bigl[\varphi-1,3/4\bigr)$ and a hereditary property $\HH$ such that $p$ is an accumulation point of $\HH$?
\end{question}

\subsection{Transformation of the weight vector}

For each of the hereditary properties $\hword$, $\hkword{k}$ and $\hcomp$, we can compute the optimal weight vector $\x$ for each CRG. 
The cases of $\hword$ and $\hkword{k}$ are trivial. 
In the case of $\hword$, the vector is $\frac{1}{3}\bigl[1,1,1\bigr]$ for $p\in \bigl(0,3/4\bigr)$ and $\bigl[1\bigr]$ for $p\in\bigl(3/4,1\bigr)$.
In the case of $\hkword{k}$ for $k\geq 2$, the vector is $\bigl[1\bigr]$ for $p\in\bigl(0,1/2\bigr)\cup\bigl(1/2,1\bigr)$.
The most interesting case is $\hcomp$.

First, we note that a direct calculation and simple trigonometric identities applied to \eqref{eq:hcomp:xa} which we leave to the reader, demonstrate that the sequence $x_1,x_2,\ldots,x_k$ is symmmetric and unimodal.

Next, we address the optimal weight vectors that occur at $p=p_k$. 
Recall that $\ed_{\hcomp}(p)$, for $p\in (0,3/4)$, is formed by $\bigl\{g_{\kpath{2}}(p),g_{\kpath{3}}(p),g_{\kpath{4}}(p),\ldots\bigr\}$.
The transition from $g_{\kpath{k-1}}(p)$ to $g_{\kpath{k}}(p)$ occurs at $p=p_k$, where $p_k=\frac{2\cos\theta+1}{2\cos\theta+2}$ for $\theta=\frac{2\pi}{k+1}$. 
It will be convenient to use the notation $P_k=2\cos\theta=\frac{2p_k-1}{1-p_k}$ and to let $G_k=G_{\kpath{k-1}}\bigl(p_k\bigr)=G_{\kpath{k}}\bigl(p_k\bigr)=1-\frac{g_{\kpath{k}}(p_k)}{1-p_k}$.
By the continuity of $\ed_{\hcomp}(p)$, the values of $g_{\kpath{k-1}}(p)$ and $g_{\kpath{k}}(p)$ are the same at $p=p_k$.
In order to compute ~\eqref{eq:hcomp:G}, we see that $\tan\bigl(\frac{k+1}{2}\theta\bigr) = \tan\pi = 0$ and so 
\begin{align*}
    G_k                             = \frac{2-P_k}{k+1} \qquad\Longleftrightarrow\qquad
    g_{\kpath{k}}\bigl(p_k\bigr)    = 1-p_k-\frac{3-4p_k}{k+1} .
\end{align*}

Now to the values of $\bigl[x_a\bigr]_{a=1}^{k-1}$ for $\kpath{k-1}$. 
Returning to~\eqref{eq:hcomp:xa},
\begin{align*}
    x_a     &=  \frac{G_k}{2-P_k}\Bigl[1-\cos\bigl(a\theta\bigr)-\sin\bigl(a\theta\bigr)\tan\bigl(\tfrac{k}{2}\theta\bigr)\Bigr] \\
            &=  \frac{1}{k+1}\Bigl[1-\cos\bigl(a\theta\bigr)-\sin\bigl(a\theta\bigr)\tan\bigl(\pi-\tfrac{\pi}{k}\bigr)\Bigr] \\
            &=  \frac{1}{k+1}\Bigl[1-\cos\bigl(a\theta\bigr)+\sin\bigl(a\theta\bigr)\tan\bigl(\theta/2\bigr)\Bigr] \\
            &=  \frac{1}{k+1}\Bigl[1-\cos\bigl(a\theta\bigr)+\sin\bigl(a\theta\bigr)\frac{1-\cos\theta}{\sin\theta}\Bigr] \\
    x_1=x_{k-1}     &=  \frac{2-2\cos\theta}{k+1} = \frac{2-P_k}{k+1} ,
\end{align*}
which happens to be $G_k$.

Now consider the optimal vector $\bigl[y_a\bigr]_{a=1}^k$ for $\kpath{k}$ and $p\in\bigl(p_k,p_{k+1}\bigr]$, but taking the limit as 
$p\rightarrow p_k^{+}$, we can use 
\begin{align*}
    y_a     &=  \frac{G_k}{2-P_k}\Bigl[1-\cos\bigl(b\theta\bigr)-\sin\bigl(b\theta\bigr)\tan\bigl(\tfrac{k+1}{2}\theta\bigr)\Bigr] \\
            &=  \frac{1}{k+1}\Bigl[1-\cos\bigl(b\theta\bigr)\Bigr] \\
    y_1=y_k     &=  \frac{1-\cos\theta}{k+1} = \frac{2-P_k}{2(k+1)} .
\end{align*}
In fact, by applying trigonometric identities, the reader can see that $y_a=\frac{1}{2}\bigl(x_{a-1}+x_a\bigr)$, as long as $x_0=0$.
In other words, $\y=\frac{1}{2}\bigl[x_1,x_2,\ldots,x_{k-1},0\bigr]+\frac{1}{2}\bigl[0,x_1,x_2,\ldots,x_{k-1}\bigr]$.

As a special note, we observe that the optimal vectors at $p=p_5=2/3$, are $\bigl[\frac{1}{6},\frac{1}{3},\frac{1}{3},\frac{1}{6}\bigr]$ and $\bigl[\frac{1}{12},\frac{1}{4},\frac{1}{3},\frac{1}{4},\frac{1}{12}\bigr]$.

\subsection{Simplified hereditary property with the same edit distance function.}

Recall that our proof of Theorem~\ref{thm:comp} used the fact that
$$
    \hcomp\subset\forb\bigl(\bigl\{F_1',F_2',C_5=\overline{C_5},\overline{C_6},\overline{C_7},\ldots\bigr\}\bigr) .
$$

In fact, we can use a weaker property to obtain the same edit distance function for $\hcomp$ by observing that 
$$
    \hcomp\subset\forb\bigl\{F_1',F_2',\overline{C_6},\overline{C_{14}},\overline{C_{30}},\ldots,\overline{C_{2^\ell-2}},\ldots\bigr\} .
$$
The same CRGs are forbidden from the hereditary property on the right-hand side because, as long as $\bigl\lceil j/2\bigr\rceil\leq k\leq j$, the graph $\overline{C_j}$ embeds in a CRG with all $k$ vertices white and all edges black, except for a gray spanning $k$-cycle. 

\subsection{$p$-core classification}

In the proof of Lemma~\ref{lem:kpath}, we showed that $\kpath{k}$ is $p$-core for $p\in\bigl(p_k,p_{2k-1}\bigr)$.
This was more than we needed for Theorem~\ref{thm:comp}. 
We ask whether the property of being $p$-core is monotone in $p$.
\begin{question}
    If $K$ is $p_0$-core for some $p_0>1/2$, then is it $p$-core for all $p\in\bigl[p_0,1\bigr)$? 
    Equivalently, if $K$ is $p_0$-core for some $p_0<1/2$, then is it $p$-core for all $p\in \bigl(p,p_0\bigr]$?
\end{question}

We do believe this is true in the special case of $\kpath{k}$.
\begin{conjecture}
    For all $k\geq 3$, the CRG $\kpath{k}$ is $p$-core for all $p\in\bigl(p_k,1\bigr)$. 
\end{conjecture}

\subsection{Edit distance for oriented graphs}
An orientation of a graph is \emph{semi-transitive} if it is acyclic (i.e., contains no directed cycles), and for any directed path $v_0 \rightarrow v_1 \rightarrow \cdots \rightarrow v_k$, either there is no edge between $v_0$ and $v_k$, or $v_i \rightarrow v_j$ is an edge for all $0 \leq i < j \leq k$. A fundamental result in the theory of word-representable graphs by Halld\'{o}rsson, Pyatkin, and the first author~\cite{HKP} states that a graph is word-representable if and only if it admits a semi-transitive orientation.

In the case of $\hword$, one may ask about the maximum edit distance of an oriented graph from one admitting a semi-transitive orientation or, in the case of $\hcomp$, the maximum edit distance of an oriented graph from one admitting a transitive orientation. This was addressed to some extent by Axenovich and the second author~\cite{AM}, although there is little other literature on the edit distance of oriented graphs, and the interested reader may wish to consider alternative measures of edit distance.

\section{Conclusion}

The second author's research is partially supported by a Simons Collaboration Grant for Mathematicians, \#709641.
The second author's research is also partially supported by a Fulbright Scholarship at the University of Birmingham, UK. 
Special thanks to the University of Birmingham and the US-UK Fulbright Commission.

\appendix

\section{Chebyshev polynomials}
\label{sec:Cheb}
The Chebyshev polynomials of the first kind (see~\cite{MathWorldT}) obey the identity $\cos\bigl(n\theta\bigr)=T_{n}\bigl(\cos\theta\bigr)$.
The Chebyshev polynomials of the second kind (see~\cite{MathWorldU}) obey the identity $\sin\bigl(n\theta\bigr)=\sin\theta\cdot U_{n-1}\bigl(\cos\theta\bigr)$. \\

\begin{center}
    \renewcommand{\arraystretch}{1.1}
    \begin{tabular}{||rcl|rcl||} \hline
        \multicolumn{6}{||c||}{Chebyshev polynomials} \\ \hline
        \multicolumn{3}{||l|}{$T_n(x)$: the first kind} & \multicolumn{3}{l||}{$U_n(x)$: the second kind} \\ [1pt] \hline
        $T_0(x)$ & $=$ & $1$ & $U_0(x)$ & $=$ & $1$ \\ [1pt] \hline
        $T_1(x)$ & $=$ & $x$ & $U_1(x)$ & $=$ & $2x$ \\ [1pt] \hline
        $T_2(x)$ & $=$ & $2x^2-1$ & $U_2(x)$ & $=$ & $4x^2-1$ \\ [1pt] \hline
        $T_3(x)$ & $=$ & $4x^3-3x$ & $U_3(x)$ & $=$ & $8x^3-4x$ \\ [1pt] \hline
        $T_4(x)$ & $=$ & $8x^4-8x^2+1$ & $U_4(x)$ & $=$ & $16x^4-12x^2+1$ \\ [1pt] \hline
        $T_5(x)$ & $=$ & $16x^5-20x^3+5x$ & $U_5(x)$ & $=$ & $32x^5-32x^3+6x$ \\ [1pt] \hline
        $T_6(x)$ & $=$ & $32x^6-48x^4+18x^2-1$ & $U_6(x)$ & $=$ & $64x^6-80x^4+24x^2-1$ \\ [1pt] \hline
    \end{tabular}~\\
\end{center}

\section{Functions that define $\ed_{\hcomp}(p)$}
\label{sec:gfunc}




Here we show values of $g_{\kpath{k}}(p)$ for small values of $k$. \\
\begin{center}
    \renewcommand{\arraystretch}{1.5}
    \begin{tabular}{||rcl|rcl||} \hline
        \multicolumn{3}{||c|}{$g$ in terms of $p$} 
        & \multicolumn{3}{c||}{$p_k$} \\ \hline
        $g_{\kpath{2}}(p)$ & $=$ & $\ds \frac{p}{2}$ 
        & $p_2$ & $=$ & $\ds 0$ \\ [5pt] \hline 
        $g_{\kpath{3}}(p)$ & $=$ & $\ds \frac{p}{1+2p}$ 
        & $p_3$ & $=$ & $\ds \frac{1}{2}$ \\ [5pt] \hline 
        $g_{\kpath{4}}(p)$ & $=$ & $\ds \frac{-1+3p-p^2}{2\bigl(-1+3p\bigr)}$ 
        & $p_4$ & $=$ & $\ds \frac{\sqrt{5}-1}{2}$ \\ [5pt] \hline 
        $g_{\kpath{5}}(p)$ & $=$ & $\ds \frac{-2+4p-p^3}{-4+6p+3p^2}$ 
        & $p_5$ & $=$ & $\ds \frac{2}{3}$ \\ [5pt] \hline 
        $g_{\kpath{6}}(p)$ & $=$ & $\ds \frac{-1-2p+9p^2-5p^3}{2\bigl(-1-2p+6p^2\bigr)}$ 
        & $p_6$ & $\approx$ & $\ds 0.692021$ \\ [5pt] \hline 
        $g_{\kpath{7}}(p)$ & $=$ & $\ds \frac{-2+13p-20p^2+6p^3+2p^4}{-3+18p-18p^2-4p^3}$ 
        & $p_7$ & $=$ & $\ds \frac{\sqrt{2}}{2}$ \\ [5pt] \hline 
        $g_{\kpath{8}}(p)$ & $=$ & $\ds \frac{-4+15p-9p^2-14p^3+11p^4}{2\bigl(-3+9p-10p^3\bigr)}$ 
        & $p_8$ & $\approx$ & $\ds 0.716881$ \\ [5pt] \hline 
        $g_{\kpath{9}}(p)$ & $=$ & $\ds \frac{-2-3p+36p^2-55p^3+20p^4+3p^5}{-3-6p+45p^2-40p^3-5p^4}$ 
        & $p_9$ & $\approx$ & $\ds 0.723607$ \\ [5pt] \hline 
        $g_{\kpath{{10}}}(p)$ & $=$ & $\ds \frac{-3+29p-75p^2+59p^3+10p^4-19p^5}{-4+36p-72p^2+20p^3+30p^4}$ 
        & $p_{10}$ & $\approx$ & $\ds 0.728446$ \\ [5pt] \hline 
    \end{tabular}~\\
\end{center}

Next, we show the optimal $\x$ vector for the CRGs $\kpath{k}$.
\begin{center}
    \renewcommand{\arraystretch}{1.5}
    \begin{tabular}{|r|l||} \hline
        CRG & 
        $\x$ in terms of $p$ \\ \hline
        $\kpath{2}$ 
        & $\frac{1}{2}\bigl[1,1\bigr]$ \\ \hline 
        $\kpath{3}$ 
        & $\frac{1}{1+2p}\bigl[p,1,p\bigr]$ \\  \hline 
        $\kpath{4}$
        & $\frac{1}{2(-1+3p)}\bigl[-1+2p,p,p,-1+2p\bigr]$ \\  \hline 
        $\kpath{5}$
        & $\frac{1}{-4+6p+3p^2}\bigl[-1+p+p^2,-1+2p,p^2,-1+2p,-1+p+p^2\bigr]$ \\  \hline 
        $\kpath{6}$
        & $\frac{1}{2(-1-2p+6p^2)}\bigl[p(-2+3p),-1+p+p^2,p(-1+2p),p(-1+2p),-1+p+p^2,p(-2+3p)\bigr]$ \\  \hline 
    \end{tabular}~\\
\end{center}

\section{Useful trigonometric identities for Lemma~\ref{lem:kpath}}
\label{sec:trig}

Here we include the trigonometric identities we use in the proof of  Lemma~\ref{lem:kpath}.
\begin{align}
    \sin a + \sin b         &=  2\sin\bigl(\tfrac{a+b}{2}\bigr)\cos\bigl(\tfrac{a-b}{2}\bigr) . \label{eq:sin:sums} \\
    \sin a                  &=  2\sin\bigl(\tfrac{a}{2}\bigr)\cos\bigl(\tfrac{a}{2}\bigr) . \label{eq:sin:doubleangle} \\
    \cos\bigl(a-b\bigr)     &=  \cos a\cos b+\sin a\sin b . \label{eq:cos:anglediff} \\
    \tan\Bigl(\frac{a}{2}\Bigr)    &=  \sqrt{\frac{1-\cos a}{1+\cos a}} = \frac{1-\cos a}{\sin a} = \frac{\sin a}{1+\cos a} , && \mbox{ if $a\in (0,\pi).$} \label{eq:tan:halfangle} \\
    \tan a-\tan b           &=  \tan\bigl(a-b\bigr)\bigl[1-\tan a\tan b\bigr] . \label{eq:tan:diffs} \\
    \sin\bigl(a-b\bigr)     &=  \sin a\cos b-\cos a\sin b . \label{eq:sin:anglediff}
\end{align}

\section{Eigenvalues and eigenvectors of a symmetric tridiagonal Toeplitz matrix}
\label{sec:Toeplitz}

We want to compute the eigenvalues and eigenvectors of the matrix $\widehat{\bf M}_{\kpath{k}}(P)$ which is defined to be
$$ \bigl(\widetilde{m}_{\kpath{k}}(P)\bigr)_{ab} 
    =   \begin{cases}
            -P,     &   \mbox{if $a=b$;} \\
            1,      &   \mbox{if $|a-b|=1$; and} \\
            0,      &   \mbox{otherwise.}
        \end{cases} $$

The eigenvalues are $\lambda_a=2\cos\bigl(\frac{a\pi}{k+1}\bigr)-P$ and the eigenvector corresponding to $\lambda_a$ is ${\bf w}_a$ where
$$ \bigl(w_a\bigr)_b = \sqrt{\frac{2}{k+1}} \sin\biggl(\frac{a\pi}{k+1}\cdot b\biggr) . $$

Proposition~\ref{prop:norm} establishes that each ${\bf w}_a$ is a unit vector and establishes that ${\bf w}_a^T{\bf 1}=0$ if $a$ is even and that ${\bf w}_a^T{\bf 1}=\sqrt{\frac{2}{k+1}}\, \frac{1+\cos\bigl(\frac{a\pi}{k+1}\bigr)}{\sin\bigl(\frac{a\pi}{k+1}\bigr)}$ if $a$ is odd.
\begin{proposition}
    For any integers $k$ and $a$, 
    \begin{align*}
        \sum_{b=1}^k\sin^2\Bigl(\frac{a\pi}{k+1}\cdot b\Bigr)   
        &=  (k+1)/2 \qquad \mbox{and} \\
        \sum_{b=1}^k\sin\Bigl(\frac{a\pi}{k+1}\cdot b\Bigr)     
        &=  \begin{cases}
                0,  &\mbox{ if $a$ is even;} \\
                \frac{1+\cos\bigl(\frac{a\pi}{k+1}\bigr)}{\sin\bigl(\frac{a\pi}{k+1}\bigr)},    &\mbox{ if $a$ is odd.}
            \end{cases}
    \end{align*}
    \label{prop:norm}
\end{proposition}

\begin{proof}
    Observe that $\sin 0=0$ and so we can start the sums at $b=0$. 
    \begin{align*}
        \sum_{b=0}^k\sin^2\Bigl(\frac{a\pi}{k+1}\cdot b\Bigr) 
        &=  \frac{k+1}{2} - \sum_{b=0}^k \left[\exp\Bigl\{\frac{2a\pi i}{k+1}\cdot b\Bigr\}+\exp\Bigl\{-\frac{2a\pi i}{k+1}\cdot b\Bigr\}\right] \\
        &=  \frac{k+1}{2} - \sum_{b=0}^k \left[\frac{1-\exp\bigl\{\frac{2a\pi i}{k+1}\cdot (k+1)\bigr\}}{1-\exp\bigl\{\frac{2a\pi i}{k+1}\bigr\}}+\frac{1-\exp\bigl\{-\frac{2a\pi i}{k+1}\cdot (k+1)\bigr\}}{1-\exp\bigl\{-\frac{2a\pi i}{k+1}\bigr\}}\right] \\
        &=  \frac{k+1}{2} .
    \end{align*}
    \begin{align*}
        \sum_{b=0}^k\sin\Bigl(\frac{a\pi}{k+1}\cdot b\Bigr) 
        &=  \frac{1}{2i}\sum_{b=0}^k \left[\exp\Bigl\{\frac{a\pi i}{k+1}\cdot b\Bigr\}-\exp\Bigl\{-\frac{a\pi i}{k+1}\cdot b\Bigr\}\right] \\
        &=  \frac{1}{2i}\left[\frac{1-\exp\{a\pi i\}}{1-\exp\bigl\{\frac{a\pi i}{k+1}\bigr\}}-\frac{1-\exp\{-a\pi i\}}{1-\exp\bigl\{-\frac{a\pi i}{k+1}\bigr\}}\right] .
    \end{align*}
    If $a$ is even, then $\exp\{a\pi i\}=\exp\{-a\pi i\}=1$ and $\sum_{b=0}^k\sin\Bigl(\frac{a\pi}{k+1}\cdot b\Bigr)=0$. 
    If $a$ is even, then $\exp\{a\pi i\}=\exp\{-a\pi i\}=-1$ and
    \begin{align*}
        \sum_{b=0}^k\sin\Bigl(\frac{a\pi}{k+1}\cdot b\Bigr) 
        &=  \frac{1}{2i}\left[\frac{2}{1-\exp\bigl\{\frac{a\pi i}{k+1}\bigr\}}-\frac{2}{1-\exp\bigl\{-\frac{a\pi i}{k+1}\bigr\}}\right] \\
        &=  \frac{1}{i}\left[\frac{\exp\bigl\{\frac{a\pi i}{k+1}\bigr\}-\exp\bigl\{\frac{-a\pi i}{k+1}\bigr\}}{\bigl(1-\exp\bigl\{\frac{a\pi i}{k+1}\bigr\}\bigr)\bigl(1-\exp\bigl\{-\frac{a\pi i}{k+1}\bigr\}\bigr)}\right] \\
        &=  \frac{\sin\bigl(\frac{a\pi}{k+1}\bigr)}{1-\cos\bigl(\frac{a\pi}{k+1}\bigr)}
        =  \frac{1+\cos\bigl(\frac{a\pi}{k+1}\bigr)}{\sin\bigl(\frac{a\pi}{k+1}\bigr)} .
    \end{align*}
\end{proof}

To verify the eigenvectors, let us compute $\widehat{\bf M}_{\kpath{k}}(P)\cdot{\bf w}_a$.
For any $b\in\{1,\ldots,k\}$
\begin{align*}
    \Bigl(\widehat{\bf M}_{\kpath{k}}(P)\cdot{\bf w}_a\Bigr)_b 
    &= \sqrt{\frac{2}{k+1}} \left[\sin\Bigl(\frac{a\pi}{k+1}\,(b-1)\Bigr) - P \sin\Bigl(\frac{a\pi}{k+1}\, b\Bigr) + \sin\Bigl(\frac{a\pi}{k+1}\,(b+1)\Bigr)\right] \\
    &=  \left[\frac{\sin\bigl(\frac{a\pi}{k+1}\,b - \frac{a\pi}{k+1}\bigr) + \sin\bigl(\frac{a\pi}{k+1}\,b + \frac{a\pi}{k+1}\bigr)}{\sin\bigl(\frac{a\pi}{k+1}\, b\bigr)} - P\right] \sqrt{\frac{2}{k+1}} \sin\Bigl(\frac{a\pi}{k+1}\, b\Bigr) \\
    &=  \left[\frac{2\sin\bigl(\frac{a\pi}{k+1}\,b\bigr)\,\cos\bigl(\frac{a\pi}{k+1}\bigr)}{\sin\bigl(\frac{a\pi}{k+1}\, b\bigr)} - P\right] \sqrt{\frac{2}{k+1}} \sin\Bigl(\frac{a\pi}{k+1}\, b\Bigr) \\
    &=  \biggl[2\cos\Bigl(\frac{a\pi}{k+1}\Bigr) - P\biggr] \sqrt{\frac{2}{k+1}} \sin\Bigl(\frac{a\pi}{k+1}\, b\Bigr),
    \end{align*}
as required.

\section{Plots of edit distance functions}

\begin{figure}[ht]
    \begin{subfigure}[t]{0.45\textwidth}
        \begin{tikzpicture}[domain=0:1,xscale=4,yscale=8]

            \useasboundingbox (-0.32,-0.09) rectangle (1.15,0.37);
  
            \def\xa{0-0.02}
            \def\xb{1+0.02}
            \def\ya{-0.01}
            \def\yb{0.25+0.01}
            \def\N{100} 
  
            \draw[xstep=0.25,ystep=0.05,ultra thin, color=gray!20]
                    (\xa,\ya) grid (\xb,\yb);
            \draw[->]
                    (\xa,0) -- (\xb,0) node[right] {$p$};
            \draw[->]
                    (0,\ya) -- (0,\yb) node[above,yshift=5pt] {$\ed_{\hword}(p)$};
  
            \draw[] 
          	 node[below left,xshift=-4pt,yshift=-5pt] at ( 0,   0) {$0$}
          	 node[below,yshift=-5pt] at ( 0.25, 0) {$0.25$}
          	 node[below,yshift=-5pt] at ( 0.5, 0) {$0.50$}
          	 node[below,yshift=-5pt] at ( 0.75, 0) {$0.75$}
          	 node[below,yshift=-5pt] at ( 1,   0) {$1$};
            \draw[] 
              node[left,xshift=-4pt] at ( 0,  0.25) {$0.25$}
          	 node[left,xshift=-4pt] at ( 0,  0.2) {$0.2$}
          	 node[left,xshift=-4pt] at ( 0,  0.15) {$0.15$}
          	 node[left,xshift=-4pt] at ( 0,  0.1) {$0.1$}
          	 node[left,xshift=-4pt] at ( 0,  0.05) {$0.05$};
    
            \draw[plt,samples=\N,domain=0:0.75] 
                plot(\x,{\x/3}) 
                node[above left] at (0.49,0.14) {$p/3$};
            \draw[plt,samples=\N,domain=0.75:1] 
                plot(\x,{1-\x})
                node[above right] at (0.83,0.14) {$1-p$};
        \end{tikzpicture}
    \end{subfigure}\ \hfill\
    \begin{subfigure}[t]{0.45\textwidth}
        \begin{tikzpicture}[domain=0:1,scale=4]

            \useasboundingbox (-0.32,-0.18) rectangle (1.15,0.75);
  
            \def\xa{0-0.02}
            \def\xb{1+0.02}
            \def\ya{-0.02}
            \def\yb{0.5+0.02}
            \def\N{100} 
  
            \draw[xstep=0.25,ystep=0.125,ultra thin, color=gray!20]
                (\xa,\ya) grid (\xb,\yb);
            \draw[->]
                (\xa,0) -- (\xb,0) node[right] {$p$};
            \draw[->]
                (0,\ya) -- (0,\yb) node[above,yshift=5pt] {$\ed_{\hkword{k}}(p)$};
  
            \draw[] 
          	 node[below left,xshift=-4pt,yshift=-5pt] at ( 0,   0) {$0$}
          	 node[below,yshift=-5pt] at ( 0.25, 0) {$0.25$}
          	 node[below,yshift=-5pt] at ( 0.5, 0) {$0.50$}
          	 node[below,yshift=-5pt] at ( 0.75, 0) {$0.75$}
          	 node[below,yshift=-5pt] at ( 1,   0) {$1$};
            \draw[] 
          	 node[left,xshift=-4pt] at ( 0,  0.5) {$0.5$}
          	 node[left,xshift=-4pt] at ( 0,  0.375) {$0.375$}
          	 node[left,xshift=-4pt] at ( 0,  0.25) {$0.25$}
          	 node[left,xshift=-4pt] at ( 0,  0.125) {$0.125$};
    
            \draw[plt,samples=\N,domain=0:0.5] 
                plot(\x,{\x}) 
                node[above left] at (0.30,0.26) {$p$};
            \draw[plt,samples=\N,domain=0.5:1] 
                plot(\x,{1-\x})
                node[above right] at (0.70,0.26) {$1-p$};
        \end{tikzpicture}
    \end{subfigure}
    \caption{Plots of $\ed_{\hword}(p)$ and $\ed_{\hkword{k}}(p)$ for $k\geq 2$. Note the different scale on the $y$-axis.}
\end{figure}
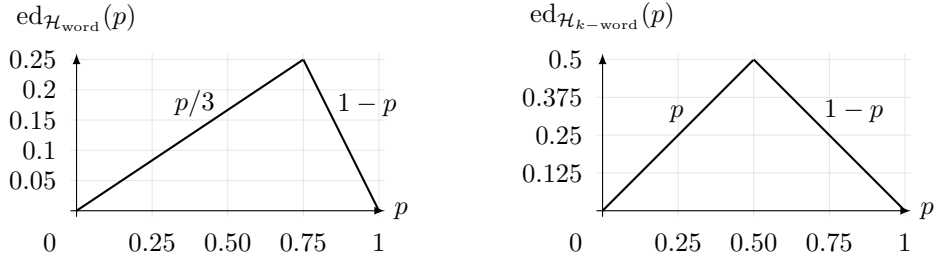

\begin{figure}[ht]
    \begin{subfigure}[t]{0.45\textwidth}
        \begin{tikzpicture}[domain=0:1,xscale=4,yscale=8]

            \useasboundingbox (-0.32,-0.10) rectangle (1.15,0.40);
  
            \def\xa{0-0.02}
            \def\xb{1+0.02}
            \def\ya{-0.01}
            \def\yb{5/18+0.01}
            \def\N{200} 
  
            \draw[xstep=1/6,ystep=1/18,ultra thin, color=gray!20]
                (\xa,\ya) grid (\xb,\yb);
            \draw[->]
                (\xa,0) -- (\xb,0) node[right] {$p$};
            \draw[->]
                (0,\ya) -- (0,\yb) node[above,yshift=5pt] {$\ed_{\hcomp}(p)$};
  
            \draw[] 
          	 node[below left,xshift=-4pt,yshift=-5pt] at ( 0,   0) {$0$}
          	 node[below,yshift=-5pt] at ( 1/6, 0) {$1/6$}
          	 node[below,yshift=-5pt] at ( 1/3, 0) {$1/3$}
          	 node[below,yshift=-5pt] at ( 1/2, 0) {$1/2$}
              node[below,yshift=-5pt] at ( 2/3, 0) {$2/3$}
          	 node[below,yshift=-5pt] at ( 5/6, 0) {$5/6$}
          	 node[below,yshift=-5pt] at ( 1,   0) {$1$};
            \draw[] 
          	 node[left,xshift=-4pt] at ( 0,  5/18) {$5/18$}
          	 node[left,xshift=-4pt] at ( 0,  4/18) {$2/9$}
          	 node[left,xshift=-4pt] at ( 0,  3/18) {$1/6$}
          	 node[left,xshift=-4pt] at ( 0,  2/18) {$1/9$}
          	 node[left,xshift=-4pt] at ( 0,  1/18) {$1/18$};
    
            \draw[plt,samples=\N,domain=0:0.5] 
                plot(\x,{\x/2}) 
                node[above left] at (0.26,0.10) {$p/2$};
            \draw[plt,samples=\N,domain=0.5:0.618] 
                plot(\x,{\x/(1+2*\x)}); 
            \draw[plt,samples=\N/2,domain=0.618:0.666] 
                plot(\x,{(-1+3*\x-\x^2)/(-2+6*\x)}); 
            \draw[plt,samples=\N/2,domain=0.667:0.692] 
              plot(\x,{(-2+4*\x-\x^3)/(-4+6*\x+3*\x^2)}); 
            \draw[plt,samples=\N,domain=0.692:0.707] 
                plot(\x,{(-1-2*\x+9*\x^2-5*\x^3)/(-2-4*\x+12*\x^2)});
            \draw[plt,samples=\N,domain=0.707:0.717] 
                plot(\x,{(-2+13*\x-20*\x^2+6*\x^3+2*\x^4)/(-3+18*\x-18*\x^2-4*\x^3)});
            \draw[plt,samples=\N,domain=0.717:0.724] 
                plot(\x,{(-4+15*\x-9*\x^2-14*\x^3+11*\x^4)/(-6+18*\x-20*\x^3)});
            \draw[plt,samples=\N,domain=0.724:0.729] 
                plot(\x,{(-2-3*\x+36*\x^2-55*\x^3+20*\x^4+3*\x^5)/(-3-6*\x+45*\x^2-40*\x^3-5*\x^4)});
            \draw[plt,samples=\N,domain=0.729:0.732] 
                plot(\x,{(-3+29*\x-75*\x^2+59*\x^3+10*\x^4-19*\x^5)/(-4+36*\x-72*\x^2+20*\x^3+30*\x^4)});
            \draw[plt,samples=\N,domain=0.732:0.75] 
                plot(\x,{((0.25-0.262)/(0.75-0.729))*(\x-0.75)+0.25});
            \draw[plt,samples=\N,domain=0.75:1] 
                plot(\x,{1-\x})
                node[above right] at (0.87,0.10) {$1-p$};
        \end{tikzpicture}
    \end{subfigure}\ \hfill\
    \begin{subfigure}[t]{0.45\textwidth}
        \begin{tikzpicture}[domain=0.5:0.75,xscale=15,yscale=70]

            \useasboundingbox (0.41,0.238) rectangle (0.79,0.295);

            \def\xa{0.5-0.005}
            \def\xb{0.75+0.005}
            \def\ya{0.25-0.002}
            \def\yb{5/18+0.002}
            \def\N{100} 
  
            \draw[xstep=0.05,ystep=1/72,ultra thin, color=gray!20]
                (\xa,\ya) grid (\xb,\yb);
            \draw[->]
                (\xa,0.25) -- (\xb,0.25) node[right] {$p$};
            \draw[->]
                (0.5,\ya) -- (0.5,\yb) node[above,yshift=5pt] {$\ed_{\hcomp}(p)$};
  
            \draw[] 
          	 node[below,yshift=-5pt] at ( 0.5, 0.25) {$0.5$}
          	 node[below,yshift=-5pt] at ( 0.55, 0.25) {$0.55$}
          	 node[below,yshift=-5pt] at ( 0.6, 0.25) {$0.6$}
          	 node[below,yshift=-5pt] at ( 0.65, 0.25) {$0.65$}
          	 node[below,yshift=-5pt] at ( 0.7, 0.25) {$0.7$}
          	 node[below,yshift=-5pt] at ( 0.75,   0.25) {$0.75$};
            \draw[] 
          	 node[left,xshift=-4pt] at ( 0.5,  5/18) {$5/18$}
          	 node[left,xshift=-4pt] at ( 0.5,  19/72) {$19/72$}
          	 node[left,xshift=-4pt] at ( 0.5,  1/4) {$1/4$};
    
            \foreach \k in {3,...,100}{
                \pgfmathsetmacro{\csk}{cos(360/(\k+1))}
                \pgfmathsetmacro{\xcor}{1-(1/(2*\csk+2))}
                \pgfmathsetmacro{\ycor}{1-\xcor+(4*\xcor-3)/(\k+1)}
                \draw node[smvtx] at (\xcor,\ycor) {};
            }
        \end{tikzpicture}
    \end{subfigure}
    \caption{Plot of $\ed_{\hcomp}(p)$ and a plot of $\Bigl\{\bigl(p_k,\ed_{\hcomp}(p)\bigr)\Bigr\}_{k\geq 2}$.}
\end{figure}

\begin{thebibliography}{99}
    \bibitem{AFKSz} Noga Alon, Eldar Fischer, Michael Krivelevich, and M\'ari\'o Szegedy, Efficient testing of large graphs. \textit{Combinatorica} \textbf{20} (2000), no. 4, 451--476.
    \bibitem{AS} Noga Alon and Uri Stav, What is the furthest graph from a hereditary property?. \textit{Random Structures Algorithms} \textbf{33} (2008), no. 1, 87--104.
    \bibitem{AKM} Maria Axenovich, Andr\'e K\'ezdy, and Ryan Martin, On the editing distance of graphs. \textit{J. Graph Theory} \textbf{58} (2008), no. 2, 123--138.
    \bibitem{AM} Maria Axenovich and Ryan R. Martin, Multicolor and directed edit distance. \texttt{https://arxiv.org/abs/1106.2870}, accessed 20 Apr 2026. 
    \bibitem{BM} J\'ozsef Balogh and Ryan Martin, Edit distance and its computation. \textit{Electron. J. Combin.} \textbf{15} (2008), no. 1, Research Paper 20, 27 pp.
    \bibitem{CMM} Christopher Cox, Ryan R. Martin, and Daniel McGinnis, Accumulation points of the edit distance function. \textit{Discrete Math.} \textbf{345} (2022), no. 7, Paper No. 112857, 17pp.
    \bibitem{ER} Paul Erd\H{o}s and Alfr\'{e}d R\'{e}nyi, On random graphs. I. \textit{Publ. Math. Debrecen} \textbf{6} (1959), 290--297.
    \bibitem{Gal} Tibor Gallai, Transitiv orientierbare Graphen. \textit{Acta Math. Acad. Sci. Hungar.} \textbf{18} (1967), 25--66.
    \bibitem{Gil} Edgar N. Gilbert, Random graphs. \textit{Ann. Math. Statist.} \textbf{30} (1959), 1141--1144.
    \bibitem{GKP} Marc Glen, Sergey Kitaev, and Artem Pyatkin, On the representation number of a crown graph. \textit{Discrete Appl. Math.} \textbf{244} (2018), 89--93.    
 \bibitem{HKP} Magn\'us M. Halld\'orsson, Sergey Kitaev, and Artem Pyatkin, Semi-transitive orientations and word-representable graphs. \textit{Discrete Appl. Math.} \textbf{201} (2016), 164–171.    
    \bibitem{HHMO} Zion Hefty, Paul Horn, Colby Muir, and Andrew Owens, Word-representation numbers of graphs: Bottlenecks and bounds, J. Combin. Th. Ser. A, \textbf{223} (2026), Paper No. 106215, 25pp.
    \bibitem{Kit} Sergey Kitaev, A comprehensive introduction to the theory of word-representable graphs. \textit{Developments in Language Theory}, 36--67. Lecture Notes in Comput. Sci., \textbf{10396} Springer, Cham, 2017.
    \bibitem{KP} Sergey Kitaev and Artem Pyatkin, On representable graphs. \textit{J. Autom. Lang. Comb.} \textbf{13} (2008), no. 1, 45--54.
    \bibitem{KS08} Sergey Kitaev and Steven Seif, Word problem of the Perkins semigroup
via directed acyclic graphs. {\em Order} {\bf 25} (2008) 3, 177--194.
\bibitem{kitaev2015}
Sergey Kitaev and Vadim Lozin, \emph{Words and Graphs},
Springer Monographs in Mathematics, Springer, Cham, 2015.
    \bibitem{MT} Edward Marchant and Andrew Thomason, Extremal graphs and multigraphs with two weighted colours. \textit{Fete of combinatorics and computer science}, 239–286. Bolyai Soc. Math. Stud., \textbf{20} J\'anos Bolyai Mathematical Society, Budapest, 2010.
    \bibitem{Mar} Ryan Martin, The edit distance function and symmetrization, \textit{Electron. J. Combin.} \textbf{20} (2013), no. 3, Paper 26, 25pp. 
    \bibitem{MR} Ryan R. Martin and Alex W.N. Riasanovsky, On the edit distance function of the random graph. \textit{Combin. Probab. Comput.} \textbf{31} (2022), no. 2, 345--367.
    \bibitem{M} Ryan R. Martin, The edit distance in graphs: methods, results, and generalizations. \textit{Recent trends in combinatorics}, 31–62. \textit{IMA Vol. Math. Appl.}, \textbf{159} Springer, [Cham], 2016.
    \bibitem{MathWorldT} Weisstein, Eric W. ``Chebyshev Polynomial of the First Kind.'' From MathWorld--A Wolfram Resource. \texttt{https://mathworld.wolfram.com/ChebyshevPolynomialoftheFirstKind.html}, accessed 14 Apr 2026. 
    \bibitem{MathWorldU} Weisstein, Eric W. ``Chebyshev Polynomial of the Second Kind.'' From MathWorld--A Wolfram Resource. \texttt{https://mathworld.wolfram.com/ChebyshevPolynomialoftheSecondKind.html}, accessed 14 Apr 2026.
\end{thebibliography}
\end{document}